\documentclass[11pt]{amsart}
\usepackage[letterpaper,margin=1in]{geometry} 
\usepackage[colorlinks,
             linkcolor=blue,
             citecolor=black!75!red,
             pdfproducer={pdfLaTeX},
             pdfpagemode=None,
             bookmarksopen=true
             bookmarksnumbered=true]{hyperref}
\usepackage[matrix,arrow,ps,color,line,curve,frame,all]{xy}

\usepackage{tikz}
\usetikzlibrary{arrows}

\title{Some algebras having relations like those for the 4-dimensional Sklyanin algebras}

\author[Alex Chirvasitu]{Alex Chirvasitu}
\author[S. Paul Smith]{S. Paul Smith}

\address{Department of Mathematics, Box 354350, University of  Washington, Seattle, WA 98195,USA.}

\email{chirva@math.washington.edu, smith@math.washington.edu}



\usepackage{amsmath,amsthm,stmaryrd}
\usepackage{cleveref}

\newtheorem{lemma}{Lemma}[section]
\newtheorem{theorem}[lemma]{Theorem}
\newtheorem{proposition}[lemma]{Proposition}
\newtheorem{corollary}[lemma]{Corollary}

\theoremstyle{definition} 

\newtheorem{definitionnodiamond}[lemma]{Definition}
\newtheorem{examplenodiamond}[lemma]{Example}
\newtheorem{remarknodiamond}[lemma]{Remark}

\numberwithin{equation}{section}

\newcounter{stepofproof}

\crefname{section}{Section}{Sections}
\crefformat{section}{#2Section~#1#3} 
\Crefformat{section}{#2Section~#1#3} 

\crefname{subsection}{}{Subsections}
\crefformat{subsection}{\S#2#1#3} 
\Crefformat{subsection}{\S#2#1#3}

\crefname{definition}{Definition}{Definitions}
\crefformat{definition}{#2Definition~#1#3} 
\Crefformat{definition}{#2Definition~#1#3} 

\crefname{example}{Example}{Examples}
\crefformat{example}{#2Example~#1#3} 
\Crefformat{example}{#2Example~#1#3} 

\crefname{table}{Table}{Tables}
\crefformat{table}{#2Table~#1#3} 
\Crefformat{table}{#2Table~#1#3}

\crefname{examplenodiamond}{Example}{Examples}
\crefformat{examplenodiamond}{#2Example~#1#3} 
\Crefformat{examplenodiamond}{#2Example~#1#3} 

\crefname{remark}{Remark}{Remarks}
\crefformat{remark}{#2Remark~#1#3} 
\Crefformat{remark}{#2Remark~#1#3} 

\crefname{remarknodiamond}{Remark}{Remarks}
\crefformat{remarknodiamond}{#2Remark~#1#3} 
\Crefformat{remarknodiamond}{#2Remark~#1#3} 

\crefname{convention}{Convention}{Conventions}
\crefformat{convention}{#2Convention~#1#3} 
\Crefformat{convention}{#2Convention~#1#3} 

\crefname{lemma}{Lemma}{Lemmas}
\crefformat{lemma}{#2Lemma~#1#3} 
\Crefformat{lemma}{#2Lemma~#1#3} 

\crefname{proposition}{Proposition}{Propositions}
\crefformat{proposition}{#2Proposition~#1#3} 
\Crefformat{proposition}{#2Proposition~#1#3} 

\crefname{corollary}{Corollary}{Corollaries}
\crefformat{corollary}{#2Corollary~#1#3} 
\Crefformat{corollary}{#2Corollary~#1#3} 

\crefname{theorem}{Theorem}{Theorems}
\crefformat{theorem}{#2Theorem~#1#3} 
\Crefformat{theorem}{#2Theorem~#1#3} 

\crefname{assumption}{Assumption}{Assumptions}
\crefformat{assumption}{#2Assumption~#1#3} 
\Crefformat{assumption}{#2Assumption~#1#3} 

\crefname{equation}{}{}
\crefformat{equation}{(#2#1#3)} 
\Crefformat{equation}{(#2#1#3)}

\crefname{align}{}{}
\crefformat{align}{(#2#1#3)} 
\Crefformat{align}{(#2#1#3)}

\crefname{proofstep}{Step}{Steps}
\crefformat{proofstep}{#2Step~#1#3} 
\Crefformat{proofstep}{#2Step~#1#3}

\newcommand\arXiv[1]{\href{http://arxiv.org/abs/#1}{\nolinkurl{arXiv:#1}}}
\newcommand\MRnumber[1]{\href{http://www.ams.org/mathscinet-getitem?mr=#1}{\nolinkurl{MR#1}}}
\newcommand\DOI[1]{\href{http://dx.doi.org/#1}{\nolinkurl{DOI:#1}}}
\newcommand\MAILTO[1]{\href{mailto:#1}{\nolinkurl{#1}}}

\usepackage{amsfonts,amssymb}

\newcommand\bG{\mathbb G}

\newcommand\bP{\mathbb P}

\newcommand\bS{\mathbb S}
\newcommand\bT{\mathbb T}

\newcommand\bX{\mathbb X}

\newcommand\cO{\mathcal O}

\newcommand\id{\mathrm{id}}

\renewcommand\lim{\varprojlim}

\def\CC{{\mathbb C}}

\def\NN{{\mathbb N}}

\def\PP{{\mathbb P}}

\def\ZZ{{\mathbb Z}}

\def\bfx{{\bf x}}

\def\GL{\operatorname {GL}}

\def\pr{{\operatorname {pr}}}

\def\Aut{\operatorname{Aut}}

\def\gr{{\sf gr}}

\def\Proj{\operatorname{Proj}}

\def\QGr{\operatorname{\sf QGr}}
\def\rank{\operatorname{rank}}

\def\a{\alpha}
\def\b{\beta}
\def\c{\gamma}
\def\d{\delta}

\def\l{\lambda}

\def\s{\sigma}

\def\ve{\varepsilon}

\def\G{\Gamma}

\def\fp{{\mathfrak p}}

\def\fP{{\mathfrak P}}
\def\fsl{{\mathfrak s}{\mathfrak l}}

\def\fgl{{\mathfrak g}{\mathfrak l}}

\def\sT{{\sf T}}

\numberwithin{equation}{section}

\makeatletter
\let\xx@thm\@thm
\AtBeginDocument{\let\@thm\xx@thm}
\makeatother

\begin{document}



\begin{abstract}
  The 4-dimensional Sklyanin algebras are a well-studied 2-parameter
  family of non-commutative graded algebras, often denoted
  $A(E,\tau)$, that depend on a quartic elliptic curve
  $E \subseteq \PP^3$ and a translation automorphism $\tau$ of
  $E$. They are graded algebras generated by four degree-one elements
  subject to six quadratic relations and in many important ways they
  behave like the polynomial ring on four indeterminates apart from
  the minor difference that they are not commutative.  They can be
  seen as ``elliptic analogues'' of the enveloping algebra of
  $\fgl(2,\CC)$ and the quantized enveloping algebras $U_q(\fgl_2)$.

  Recently, Cho, Hong, and Lau, conjectured that a certain 2-parameter
  family of algebras arising in their work on homological mirror
  symmetry consists of 4-dimensional Sklyanin algebras. This paper
  shows their conjecture is false in the generality they make it. On
  the positive side, we show their algebras exhibit features that are
  similar to, and differ from, analogous features of the 4-dimensional
  Sklyanin algebras in interesting ways.  We show that most of the
  Cho-Hong-Lau algebras determine, and are determined by the graph of
  a bijection between two 20-point subsets of the projective space
  $\PP^3$.

  The paper also examines a class of 4-generator 6-relator algebras
  admitting presentations analogous to those of the 4-dimensional
  Sklyanin algebras. This class includes the 4-dimensional Sklyanin
  algebras and most of the Cho-Hong-Lau algebras.
\end{abstract}

\maketitle

\tableofcontents
\pagenumbering{arabic}

\section{Introduction}

\subsection{}
This paper examines three families of graded algebras with four
generators and six quadratic relations.  The only commutative algebra
in these families is the polynomial ring on 4 variables.  All algebras
in these families are, like the polynomial ring on 4 variables,
generated by 4 elements subject to 6 homogeneous quadratic relations.

The members of the first of these families are denoted by
$A(\a,\b,\c)$, depending on a parameter $(\a,\b,\c) \in \Bbbk^3$ where
$\Bbbk$ is a field that will be fixed throughout the paper. They are
generated by $x_0,x_1,x_2,x_3$ subject to the relations
 \begin{equation}
 \label{Aabc}
 \begin{cases}
  x_0x_1-x_1x_0  & \; = \; \a(x_2x_3+x_3x_2) \qquad \qquad  x_0x_1+x_1x_0   \; = \; x_2x_3-x_3x_2  \\
   x_0x_2-x_2x_0  & \; = \; \b(x_3x_1+x_1x_3) \qquad \qquad  x_0x_2+x_2x_0   \; = \; x_3x_1-x_1x_3  \\
    x_0x_3-x_3x_0  & \; = \; \c(x_1x_2+x_2x_1) \qquad \qquad  x_0x_3+x_3x_0   \; = \; x_1x_2-x_2x_1.
\end{cases}
\end{equation}
Among these algebras, those for which
 \begin{equation}
 \label{constraints}
\a+\b+\c+\a\b\c=0 \qquad \text{ and} \qquad \{\a,\b,\c\} \cap \{0,\pm 1\}=\varnothing,
\end{equation}
are so starkly different from the rest that we consider them as a
separate family.  These constitute the second of our three families
and are called {\it non-degenerate 4-dimensional Sklyanin algebras.}
Algebras in the third family are denoted by $R(a,b,c,d)$, depending on
a parameter $(a,b,c,d)$ that is required to lie on the quadric
$\{ad+bc=0\}$ in the projective space $\PP^3$. They are defined in
\Cref{sect.CHL.algs1}.

The algebras $R(a,b,c,d)$ were discovered by Cho, Hong, and Lau in
their work on mirror symmetry \cite{CHL}, and the motivation for this
paper is their conjecture that these are 4-dimensional Sklyanin
algebras. We prove the conjecture false in the generality in which it
is made, but on the positive side
\begin{enumerate}
\item{} for a Zariski-dense open subset of points on the quadric
  $\{ad+bc=0\}$, $R(a,b,c,d)$ is isomorphic to $A(\a,\b,\c)$ for some
  $(\a,\b,\c)$, but $(\a,\b,\c)$ does not always satisfy the condition
  $\a+\b+\c+\a\b\c=0$;
\item{} there are two lines $\ell_1, \ell_2 \subseteq \{ad+bc=0\}$
  such that $R(a,b,c,d)$ is isomorphic to $A(\a,1,-1)$ for all
  $(a,b,c,d) \in \ell_1 \cup \ell_2 - \{\text{12 points}\}$;
\item the automorphism group of almost all $R(a,b,c,d)$ has a subgroup
  isomorphic to the Heisenberg group of order $4^3$.
\end{enumerate}

The non-degenerate Sklyanin algebras may be parametrized by pairs
$(E,\tau)$ consisting of an elliptic curve $E$ and a translation
automorphism $\tau:E \to E$. We write $A(E,\tau)$ for the Sklyanin
algebra corresponding to this data.  It is striking that the
translation automorphism for those $R(a,b,c,d)$ that are
non-degenerate Sklyanin algebras has order 4; i.e., if
$(a,b,c,d) \in \ell_1 \cup \ell_2 - \{\text{12 points}\}$, then
$R(a,b,c,d)\cong A(E,\tau)$ for some elliptic curve $E$ and some
$\tau$ having order 4 (\Cref{prop.almost.Skly,prop.special.case}).

\subsection{}
\label{sect1.20.pts}
A striking feature of the algebras $R(a,b,c,d)$ is that almost all of
them determine, and are determined by, a set of 20 points in the
product $\PP^3\times \PP^3$ of two copies of the three-dimensional
projective space.

\subsection{}
Because Sklyanin algebras, appearing first in \cite{Skl82,Skl83}, have
played such a large role in the development of non-commutative algebra
and algebraic geometry over the past thirty years (see
\cite{SS92,LS93,SSJ93,S94,VdB96} for example), it is sensible to
examine the larger class of algebras $A(\a,\b,\c)$ defined by the
``same'' relations minus the constraint $\a+\b+\c+\a\b\c=0$.

We do not undertake an exhaustive study of the algebras $A(\a,\b,\c)$
when $\a+\b+\c+\a\b\c \ne 0$ but it appears to us that there are
interesting questions about them that might be fruitfully pursued. We
mention some of these questions in \Cref{sect.qus}.

\subsection{}

We use the notation $[x,y]=xy-yx$ and $\{x,y\}=xy+yx$.

\subsection{The algebras $A(\a,\b,\c)$}

 Let  $\Bbbk$ be an arbitrary field and $\a_1,\a_2,\a_3 \in \Bbbk$.
 Define $A(\a_1,\a_2,\a_3)$, or simply $A$, to be the free algebra $\Bbbk\langle x_0,x_1,x_2,x_3\rangle$ modulo the six relations
\begin{equation}
\label{skly.relns}
[x_0,x_i]=\a_i\{x_j,x_k\}, \quad \{x_0,x_i\}=[x_j,x_k], \qquad \hbox{$(i,j,k)$ a cyclic permutation of $(1,2,3)$.}
\end{equation}
We always consider $A$ as an $\NN$-graded $\Bbbk$-algebra with
${\rm deg}\{x_0,x_1,x_2,x_3\}=1$.  Thus, $A$ is the quotient of the
free algebra $TV/(R)=\Bbbk\langle x_0,x_1,x_2,x_3\rangle/(R)$ where
$V={\rm span}\{x_0,x_1,x_2,x_3\}$ and $R \subseteq V^{\otimes 2}$ is
the linear span of the six elements in $V^{\otimes 2}$ corresponding
to the relations \Cref{skly.relns}.

\subsection{Degenerate and non-degenerate 4-dimensional Sklyanin algebras}
Suppose $\a+\b+\c+\a\b\c =0$. We call  $A(\a,\b,\c)$ a {\sf 4-dimensional Sklyanin algebra} in this case.
If, in addition, $\{\a,\b,\c\} \cap \{0,\pm 1\}=\varnothing$ we call $A(\a,\b,\c)$
a {\sf non-degenerate 4-dimensional Sklyanin algebra}.
If $\a+\b+\c+\a\b\c =0$ and $\{\a,\b,\c\} \cap \{0,\pm 1\} \ne \varnothing$ we call $A(\a,\b,\c)$ a {\sf degenerate 4-dimensional Sklyanin algebra}. 

By \cite{SS92}, non-degenerate 4-dimensional Sklyanin algebras are
noetherian domains having the same Hilbert series as the polynomial
ring in 4 variables. By \cite{SS92} and \cite{LS93}, they have excellent homological
properties. The representation theory is intimately related to the geometry of $(E \subseteq \PP^3,\tau)$.

Some degenerate 4-dimensional Sklyanin algebras are closely related to
better known algebras.  For instance, the algebra $A=A(0,0,0)$ has a
degree-one central element, $z$, such that
$A/(z-1) \cong A[z^{-1}]_0 \cong U({\mathfrak{so}}(3,\Bbbk))$, the
enveloping algebra of the Lie algebra ${\mathfrak{so}}(3,\Bbbk)$.
Similarly, if $\Bbbk=\CC$ and $\b \ne 0$, then $A=A(0,\b,-\b)$ has a
degree-two central element $\Omega$ such that
$A[\Omega^{-1}]_0 \cong U_q(\fsl(2,\CC))$, a quantized enveloping
algebra of $\fsl(2,\CC)$.

\subsection{The algebras of  Cho, Hong, and Lau}
\label{sect.CHL.algs1}
\label{ssect.CHL.alg.defn}


Let  $(a,b,c,d) \in \Bbbk^4$. We write $R(a,b,c,d)$, or simply $R$, for the free algebra $\Bbbk\langle x_1,x_2,x_3,x_4\rangle$ modulo the relations
\begin{align*}
\hbox{(R1)} \qquad \phantom{xxxxx} & ax_4x_3+bx_3x_4 +cx_3x_2+dx_4x_1=0,   \\
\hbox{(R2)} \qquad \phantom{xxxxx}  & ax_3x_2+bx_2x_3 +cx_4x_3+dx_1x_2=0,   \\
\hbox{(R3)} \qquad \phantom{xxxxx}  & ax_2x_1+bx_1x_2 +cx_1x_4+dx_2x_3=0,   \\
\hbox{(R4)} \qquad \phantom{xxxxx} & ax_1x_4+bx_4x_1 +cx_2x_1+dx_3x_4=0,   \\
\hbox{(R5)} \qquad \phantom{xxxxx} &  ax_3x_1-ax_1x_3 +cx_4^2 -cx_2^2=0,      \\
\hbox{(R6)} \qquad \phantom{xxxxx} & bx_4x_2-bx_2x_4 +dx_3^2 -dx_1^2=0.   
\end{align*}
Since $a$, $b$, $c$, and $d$, enter into the relations in a homogeneous
way, the algebra $R(a,b,c,d)$ depends only on $(a,b,c,d)$ as a point
in $\PP^3$. If we impose the condition that $ac+bd=0$, we obtain a
2-dimensional family of algebras $R(a,b,c,d)$ parametrized by a
quadric (isomorphic to $\PP^1 \times \PP^1$) in $\PP^3$.  Cho, Hong,
and Lau conjecture that, when $ac+bd=0$, $R$ is a 4-dimensional
Sklyanin algebra \cite[Conj. 8.11]{CHL}.

Although only a 1-parameter family of the $R(a,b,c,d)$ are Sklyanin
algebras, we find it remarkable that almost all of them
(Zariski-densely many, that is) have the ``same'' relations as the
4-dimensional Sklyanin algebras. We do not understand the deeper
reason for this; our proof is just a calculation. We also find it
remarkable that the translation automorphism for those that are
Sklyanin algebras has order 4---the only translation automorphisms of
a degree-four elliptic curve in $\PP^3$ that extend to automorphisms
of the ambient $\PP^3$ are the translations by the points in its
4-torsion subgroup. We do not know in what way, in the context of the
work of Cho-Hong-Lau, those $R(a,b,c,d)$ that are Sklyanin algebras
are special.

\subsection{Results about $A(\a,\b,\c)$}
Suppose $\a\b\c \ne 0$. In \Cref{sect.algebras} we show that the
Heisenberg group of order $4^3$ acts as automorphisms of
$A(\a,\b,\c)$. In \Cref{prop.isom}, we determine exactly when two of
these algebras are isomorphic to each other.

In \Cref{sect.Gamma} we give a geometric interpretation of the
relations defining $A(\a,\b,\c)$. To do this we first write
$A(\a,\b,\c)$ as $TV/(R)$, the quotient of the tensor algebra $TV$ on
a 4-dimensional vector space $V$ by the ideal generated by a
6-dimensional subspace $R$ of $V^{\otimes 2}$. We then consider
elements in $V \otimes V$ as forms of bi-degree $(1,1)$ on the product
$\PP(V^*) \times \PP(V^*)$ of two copies of projective 3-space. We now
define the closed subscheme $\G \subseteq \PP(V^*) \times \PP(V^*)$ to
be the vanishing locus of the elements in
$R$. \Cref{prop.Gamma.finite.iff} shows that $\G$ is finite if and
only if $\a\b\c\ne 0$ and $\a+\b+\c+\a\b\c \ne 0$.
\Cref{prop.Gamma.finite.iff,prop.Gamma} show that in that case
\begin{enumerate}
\item{} $\G$ consists of 20 distinct points;
\item $\G$ is the graph of a bijection between 20-point subsets of
  $\PP(V^*)$;
\item $R=\{f \in V^{\otimes 2} \; | \; f|_\G=0\}$.
\end{enumerate}

\subsection{Centers} 
\cite[Thm. 2]{Skl82} states that two explicitly given degree-two
homogeneous elements, which are denoted there by $K_0$ and $K_1$,
belong to the center of the 4-dimensional Sklyanin algebra. Although
Sklyanin writes that it is ``straightforward'' to prove these elements
central, the details are left to the reader.  We and others have found
the calculations less than straightforward.\footnote{The problem of
  showing that $K_0$ and $K_1$ are central is mentioned in a talk
  given by Tom Koornwinder at Nijmegen on 12 November 2012---see {\tt
    https://staff.science.uva.nl/t.h.koornwinder/art/sheets/SklyaninAlgebra1.pdf},
  retrieved on 01-20-2017. Koornwinder says that part of the proof is
  ``straightforward'' and appeals to the Mathematica package NCAlgebra
  4.0.4 at {\tt http://www.math.ucsd.edu/\textasciitilde ncalg/ } for
  the remainder of the proof.}  Sklyanin says that an alternative
proof can be given by using a lemma in his paper \cite{KS82} with
Kulish. Presumably, the relevant lemma is equation (5.7) in
\cite{KS82}. However, ``due to space limitations [they] do not present
[t]here the complete proof of (5.7)''.  We have been unable to find a
complete proof of \cite[Thm. 2]{Skl82} in the literature so
give a direct proof that $K_0$ and $K_1$ are central in
\Cref{prop.Skly.center} below. We do not use the same notation as
Sklyanin, so in this introduction we label the elements in
\Cref{prop.Skly.center} $\Omega_0$ and $\Omega_1$.

For most 4-dimensional Sklyanin algebras $K_0$ and $K_1$, or
equivalently $\Omega_0$ and $\Omega_1$, generate the center of the
algebra.  In sharp contrast, when $\a\b\c\ne0$ and
$\a+\b+\c+\a\b\c \ne 0$ the elements $x_0^2$, $x_1^2$, $x_2^2$, and
$x_3^2$ belong to the center of $A(\a,\b,\c)$ (\Cref{prop.moore}).

Cho, Hong, and Lau, write down two degree-two elements in $R(a,b,c,d)$
that they conjecture belong to the center of $R(a,b,c,d)$.  We verify
their conjecture in \Cref{pr.O1} and \Cref{cor.center}.

It is interesting to compare the proof of these results about the
centers to the proof that the Casimir elements in the enveloping
algebras $U(\fsl_2)$ and $U_q(\fsl_2)$ belong to the center. The
latter proofs are absolutely straightforward, whereas the computations
involved in describing the centers of $A(\a,\b,\c)$ are far less
routine because these algebras do not have a PBW basis (or,
apparently, any basis that makes computation routine). See however,
the notion of an $I$-algebra in \cite{TvdB96}.

\subsection{Some questions about $A(\a,\b,\c)$}
\label{sect.qus}
Computer calculations by Frank Moore suggest that the dimensions of
the homogeneous components $A(\a,\b,\c)_n$ are
$1,4,10,16,19,20,20,20,\ldots$ when $\a+\b+\c+\a\b\c \ne 0$. Is this
true?  If so, then for a generic linear combination $\Omega$ of the
central elements $x_0^2,x_1^2,x_2^2,x_3^2 \in A(\a,\b,\c)_2$ the
localization $A[\Omega^{-1}]_0$ is a finite dimensional algebra having
dimension 20.  We expect the representation theory of these finite
dimensional algebras is interesting.

%
%
%

We do not know if $A(\a,\b,\c)$ is a Koszul algebra (Sklyanin algebras
are) but whether it is or not its quadratic dual $A(\a,\b,\c)^!$
deserves investigation.

We show in \Cref{sect.Gamma}, when $\a+\b+\c+\a\b\c \ne 0$ and
$\a\b\c \ne 0$, that the algebra $A(\a,\b,\c)$ determines and is
determined by a configuration of 20 points in $\PP^3\times \bP^3$
that is the graph of a bijection between two 20-point elements of
$\bP^3$. We do not understand this configuration but the
representation theory of $A(\a,\b,\c)$ and these related algebras is
governed by it. The details of this are likely to be interesting and
novel.

It would be interesting to understand how the configuration of 20 points relates to the features of
$R(a,b,c,d)$ that are relevant to the work of Cho, Hong, and Lau.

\subsection{Acknowledgements}
We wish to thank Frank Moore whose computer calculations involving the
algebras $A(\a,\b,\c)$ defined in \Cref{skly.relns} were of great
assistance to us at an early stage of this project.  His calculations
showed that over certain finite fields the elements
$x_0^2,\ldots,x_3^2$ belong to the center of $A(\a,\b,\c)$ when
$\a\b\c\ne 0$ and $\a+\b+\c+\a\b\c \ne 0$.  Based on those
calculations we then proved the centrality of those elements over all
fields (\Cref{prop.moore}).

The work of the first author was partially supported by NSF grant DMS-1565226.

\section{Algebras $A(\a,\b,\c)$ with a Sklyanin-like presentation}
\label{sect.algebras}
 
 \subsection{Notation}
 Throughout this paper $\Bbbk$ denotes a field whose characteristic is
 not 2, and $i$ denotes a fixed square root of $-1$.
 
 Whenever we use parameters $\a,\b,\c \in \Bbbk$ we will assume they
 have square roots $a,b,c \in \Bbbk$.
 
 We fix a 4-dimensional $\Bbbk$-vector space $V$. Always,
 $x_0,x_1,x_2,x_3$ will denote a basis for $V$.

 \subsubsection{}
 We write $TV$ for the tensor algebra on $V$. Thus $TV$ is the free algebra $\Bbbk\langle x_0,x_1,x_2,x_3  \rangle$. We always consider $TV$ as an $\ZZ$-graded
 $\Bbbk$-algebra with $\deg(V)=1$. All the algebras in this paper are of the form $A=TV/(R)$ for various 6-dimensional subspaces $R$ of $V^{\otimes 2}$.
 
  \subsubsection{}
  Let $\a,\b,\c \in \Bbbk$. The algebra $A(\a,\b,\c)$ is the free
  algebra $TV$ modulo the relations in \Cref{Aabc}.
  
  
  \subsubsection{}
  \label{sssect.notn}
  We will often write $(\a_1,\a_2,\a_3)=(\a,\b,\c)$. In \Cref{sect.algebras} and \Cref{sect.Gamma}, 
  $a,b,c$ will denote fixed square roots of $\a,\b,\c$.   We will often write $(a_1,a_2,a_3)=(a,b,c)$. 

\subsubsection{}
Let $A$ be a $\ZZ$-graded $\Bbbk$-algebra.  We write $\Aut_{\gr}(A)$
for the group of graded $\Bbbk$-algebra automorphisms of $A$.

If $\l \in \Bbbk^\times$, $\phi_\l$ denotes the automorphism of $A$
that is multiplication by $\l^n$ on $A_n$. The map
$\Bbbk^\times \to \Aut_{\gr}(A)$, $\l \mapsto \phi_\l$, is an
injective group homomorphism whose image lies in the center.  We will
often identify $\l$ with $\phi_\l$. If $\psi \in \Aut(A)$, we will
write $\psi^m=\l$ if $\psi^m=\phi_\l$ and $\l\psi$ for $\phi_\l \psi$.

\subsubsection{}
\label{ssect.psi_j}
Suppose $a,b,c \in \Bbbk^\times$. We define
$\psi_1,\psi_2,\psi_3 \in \GL(V)$ by declaring that $\psi_i(x_j)$ is
the entry in row $\psi_i$ and column $x_j$ in \Cref{autom}.
\begin{table}[htp]
\begin{center}
\begin{tabular}{|c|c|c|c|c|c|c|c|}
\hline
 & $x_0$ &  $x_1$ &  $x_2$ &  $x_3$  
\\
\hline
$\psi_1 $  & $bc x_1$ &   $-i x_0$ &  $-ib x_3$ &  $-c x_2$ $\phantom{\Big)}$
\\
\hline
$\psi_2$  & $ac x_2$ &   $-a  x_3$ &  $-ix_0$ &  $ -ic x_1$ $\phantom{\Big)}$
\\
\hline
$\psi_3$  & $ab x_3$ &   $-ia x_2$ &  $-bx_1$ &  $ -i  x_0$ $\phantom{\Big)}$
\\
\hline
\end{tabular}
\end{center} 
\caption{Automorphisms $\psi_1,\psi_2,\psi_3$.}
\label{autom}
\end{table}
\newline
In the notation of \Cref{sssect.notn}, if  $(i,j,k)$ is a cyclic permutation of $(1,2,3)$, then
$$
\psi_i(x_0)=a_ja_kx_i, \qquad \psi_i(x_i)=-ix_0, \qquad \psi_i(x_j)=-ia_jx_k,  \qquad \hbox{and} \qquad \psi_i(x_k)=-a_kx_j.
$$

\subsubsection{The Heisenberg group of order $4^3$}
The Heisenberg group of order $4^3$ is 
$$
H_4 \; : = \; \langle \ve_1,\ve_2,\d \; | \; \ve_1^4=\ve_2^4=\d^4=1, \;  \d\ve_1=\ve_1\d, \; \ve_2\d=\d\ve_2, \; \varepsilon_1\varepsilon_2=\delta \varepsilon_2\varepsilon_1\rangle.
$$

 \subsection{} 
 \label{sect.gamma_js}

By  \cite{FO98} and \cite[pp. 64-65]{SSJ93}, for example, the Heisenberg group $H_4$ acts as graded 
$\Bbbk$-algebra automorphisms of the 4-dimensional Sklyanin algebras when $\Bbbk=\CC$. 
The next result records the fact that  $H_4$ acts as graded $\Bbbk$-algebra automorphisms of $A(\a,\b,\c)$ whenever $\a\b\c \ne 0$ and $\Bbbk$ is a field having 
square roots of $\a$, $\b$, $\c$, and $-1$.  

\begin{proposition} 
\label{prop.aut.A}
Suppose $\a\b\c \ne 0$. Fix $\nu_1,\nu_2,\nu_3 \in \Bbbk^\times$ such that $a\nu_1^2=b\nu_2^2=c\nu_3^2=-iabc$. 
\begin{enumerate}
  \item 
  The maps $\psi_1,\psi_2,\psi_3:V \to V$ in   \Cref{autom} extend to $\Bbbk$-algebra automorphisms
of $A(\a,\b,\c)$.
  \item 
  There is an injective homomorphism $H_4 \to \Aut_{\sf gr}(A)$ given by 
$$
\ve_1\mapsto  \nu_1^{-1}\psi_1, \qquad \ve_2 \mapsto  \nu_2^{-1}\psi_2, \qquad \ve_3  \mapsto \nu_3^{-1}\psi_3.
$$ 
Under this map, $\d \mapsto \phi_i$, the automorphism that is multiplication by $i^n$ on $A(\a,\b,\c)_n$. 
\item{}
The subgroup of $\Aut_{\sf gr}(A)$ generated by $\c_1:=\ve_1^2$, $\c_2:=\ve_2^2$, and $\c_3:=\ve_3^2$ is isomorphic to $\ZZ_2 \times \ZZ_2$. The value of
$\c_i(x_j)$ is the  entry in row $\c_i$ and column  $x_j$ of Table \Cref{Gamma.action} below. 
\begin{table}[htp]
\begin{center}
\begin{tabular}{|c|c|c|c|c|c|c|c|}
\hline
 & $x_0$ &  $x_1$ &  $x_2$ &  $x_3$
\\
\hline
$\c_1 $  & $x_0$ &  $x_1$ &  $-x_2$ &  $-x_3$
\\
\hline
$\c_2 $  & $x_0$ &  $-x_1$ &  $x_2$ &  $-x_3$
\\
\hline
$\c_3$  & $x_0$ &  $-x_1$ &  $-x_2$ &  $x_3$
\\
\hline
\end{tabular}
\end{center}
\vskip .12in
\caption{The action of $\ZZ_2\times \ZZ_2$ as automorphisms of $A$.}
 \label{Gamma.action}
\end{table}
\newline 
\end{enumerate}
\end{proposition}
\begin{proof}
Let   $(i,j,k)$ be a cyclic permutation of $(1,2,3)$ and let $\l_0,\l_i,\l_j,\l_k \in \Bbbk^\times$.
In \cite[Prop. 4]{Skl83}, Sklyanin observed that the linear map  $\psi:V \to V$  acting on $x_0,x_i,x_j,x_k$ as
\begin{table}[htp]
\begin{center}
\begin{tabular}{|c|c|c|c|c|c|c|c|}
\hline
 & $x_0$ &  $x_i$ &  $x_j$ &  $x_k$
\\
\hline
$\psi$  & $\l_0x_i$ &   $\l_ix_0$ &  $\l_jx_k$ &  $\l_kx_j$ 
\\
\hline
\end{tabular}
\end{center} 
\end{table}

\noindent
extends to an automorphism of  the Sklyanin algebra  if and only if  
\begin{equation}
\label{autom.conds}
\frac{\l_0\l_i}{\l_j\l_k} \; = \; -1, \qquad 
\frac{\l_0\l_j}{\l_k\l_i} \; = \; -\a_j, \quad \hbox{and} \qquad
\frac{\l_0\l_k}{\l_i\l_j} \; = \;  \a_k.
\end{equation}
A straightforward calculation shows that $\psi$ extends to an automorphism of $A(\a,\b,\c)$ without any restriction on $\a,\b,\c$ other than $\a\b\c \ne 0$ 
if and only if \Cref{autom.conds} holds. The maps $\psi_1$, $\psi_2$, and $\psi_3$, satisfy these conditions so extend to graded $\Bbbk$-algebra automorphisms of $A$.  

It is easy to check that $\psi_1\psi_2=\d\psi_2\psi_1$, $\psi_2\psi_3=\d\psi_3\psi_2$, and $\psi_3\psi_1=\d\psi_1\psi_3$.
It follows that $\ve_1\ve_2=\d\ve_2\ve_1$, $\ve_2\ve_3=\d\ve_3\ve_2$ ,and $\ve_3\ve_1=\d\ve_1\ve_3$.  

It is easy to check that $\c_1,\c_2,\c_3$ act on $x_0,x_1,x_2,x_3$ as in \Cref{Gamma.action}.   Hence $\ve_1^4=\ve_2^4=\ve_3^4=1$. 
Simple calculations show that  $\psi_1^2=-ibc\c_1$,  $\psi_2^2=-iac\c_2$,  and $\psi_3^2=-iab\c_3$, where $\c_1$, $\c_2$, and $\c_3$, are the automorphisms in \Cref{Gamma.action}.  
 We leave the rest of the proof to the reader. 
\end{proof}

 \subsubsection{}
 The maps $\c_i \in \GL(V)$ given by \Cref{Gamma.action} extend to graded $\Bbbk$-algebra  automorphisms of $A(\a,\b,\c)$ for {\it all} $\a,\b,\c \in \Bbbk$.

\subsection{}
In the next result, whose proof we omit, 
$([A,A])$ denotes the ideal in $A$ generated by all commutators $ab-ba$, $a,b \in A$. Thus,  $A/([A,A])$ is the largest commutative quotient of $A$.

\begin{proposition}
Suppose $\a\b\c \ne 0$. Let $A=A(\a,\b,\c)$.  
\begin{enumerate}
  \item 
  As a quotient of the polynomial ring $\Bbbk [x_0,x_1,x_2,x_3]$, 
  $$
\frac{A}{([A,A])} \; = \;    
\frac{\Bbbk [x_0,x_1,x_2,x_3]}{(x_1,x_2,x_3) \cap (x_0,x_2,x_3) \cap(x_0,x_1,x_3) \cap(x_0,x_1,x_2)}\,.
$$
  \item 
  As a subscheme of $\PP(V^*)$, 
  $$
\Proj\bigg(\frac{A}{([A,A])}\bigg) \; =\; \{e_0=(1,0,0,0), e_1=(0,1,0,0), e_2=(0,0,1,0), e_3=(0,0,0,1)\}.
$$
  \item 
  $A$ has exactly four graded quotients  that are polynomial rings in one variable, namely the quotients by
the ideals $(x_1,x_2,x_3)$, $(x_0,x_2,x_3)$, $(x_0,x_1,x_3)$, and $(x_0,x_1,x_2)$.
\end{enumerate}
\end{proposition}

\begin{lemma}
\label{lem.isom}
There are algebra isomorphisms
\begin{align*}
A(\a,\b,\c)  & \; \cong \; A(\b,\c,\a) \;  \cong \; A(\c,\a,\b)
\\
 & \; \cong \;   A(-\a,-\c,-\b)   \; \cong \;   A(-\b,-\a,-\c)  \; \cong \;   A(-\c,-\b,-\a) .
\end{align*}
\end{lemma}
\begin{proof}
There is an isomorphism $A(\a,\b,\c) \stackrel{\sim}{\longrightarrow} A(\b,\c,\a)$ given by $x_0 \mapsto x_0$ and $x_i \mapsto x_{i+1}$ for $i \in \{1,2,3\}=\ZZ/3$.
Similarly, $A(\b,\c,\a) \cong A(\c,\a,\b)$.
 Since
$$
\phantom{xxiii} [x_0,-x_1]=-\a\{x_3,x_2\}, \qquad \phantom{ii}  [x_0,x_3]=-\c\{x_2,-x_1\}, \qquad \phantom{ii}  [x_0,x_2]=-\b\{-x_1,x_3\}, 
$$
and
$$
\{x_0,-x_1\}=[x_3,x_2], \qquad \qquad \{x_0,x_3\}=[x_2,-x_1], \qquad \qquad \{x_0,x_2\}=[-x_1,x_3], 
$$
there is an isomorphism $A(\a,\b,\c) \stackrel{\sim}{\longrightarrow} A(-\a,-\c,-\b)$ given by $x_0 \mapsto x_0$, $x_1 \mapsto -x_{1}$, $x_2 \mapsto x_3$, and 
$x_3 \mapsto -x_2$.
\end{proof} 

\begin{proposition}
\label{prop.isom}
Suppose $\a\b\c \ne 0$ and $\a'\b'\c' \ne 0$.
Then $A(\a,\b,\c) \cong A(\a',\b',\c')$ as graded $\Bbbk $-algebras 
if and only if $(\a',\b',\c')$ is a cyclic permutation of either $(\a,\b,\c)$ or $(-\a,-\b,-\c)$.
\end{proposition}
\begin{proof}
($\Leftarrow$)
This is the content of \Cref{lem.isom}.

($\Rightarrow$)
Before starting the proof we introduce some notation.
If $(p,q,r,s)$ is a permutation of $(0,1,2,3)$ we define
$$
 \langle p,q,r,s\rangle_A  := (\mu_1\nu_1, \mu_2\nu_2,\mu_3\nu_3) \; \in \; \Bbbk^3
 $$
where $\mu_1,\mu_2,\mu_3,\nu_1,\nu_2,\nu_3$ are the unique scalars such that
\begin{align*}
[x_p,x_q] & \; = \; \mu_1\{x_r,x_s\}, \qquad \nu_1\{x_p,x_q\} \; = \; [x_r,x_s],
\\
[x_p,x_r] & \; = \; \mu_2\{x_r,x_s\}, \qquad \nu_2\{x_p,x_r\} \; = \; [x_s,x_q],
\\
[x_p,x_s] & \; = \; \mu_3\{x_r,x_s\}, \qquad \nu_3\{x_p,x_s\} \; = \; [x_q,x_r],
\end{align*}
in $A$.
It is easy to see that 
\begin{equation}
\label{angle.invariants}
\begin{cases}
\langle 0,1,2,3 \rangle_A= \langle 1,0,3,2 \rangle_A =  \langle 2,3,0,1 \rangle_A = \langle 3,2,1,0 \rangle_A = (\a_1,\a_2,\a_3) & \; \text{and}  \;
\\
\langle 0,1,3,2 \rangle_A =  \langle 1,2,3,0 \rangle_A = \langle 2,1,0,3 \rangle_A= \langle 3,1,2,0 \rangle_A =(-\a_1,-\a_3,-\a_2). &
\end{cases}
 \end{equation}
 If $ \langle p,q,r,s\rangle_A  
 = (\l_1, \l_2,\l_3)$, then 
 $\langle p,r,s,q\rangle _A= (\l_2,\,\l_3,\l_1)$. Using this and the equalities in \Cref{angle.invariants}, it is easy to compute 
  $\langle p,q,r,s\rangle_A $ 
    for all permutations $(p,q,r,s)$ of $(0,1,2,3)$.

Let's write $A=A(\a,\b,\c)$ and $B=A(\a',\b',\c')$. To distinguish the presentation of $A$ from that for $B$ we will write $x_0,x_1,x_2,x_3$ for the generators of $A$,
as in \Cref{skly.relns}, and write $x_0',x_1',x_2',x_3'$ for the generators of $B$. Thus, if $(\b_1,\b_2,\b_3)=(\a',\b',\c')$, then 
$[x_0',x_i']=\b_i\{x_j',x_k'\}$ and $\{x_0',x_i'\}=[x'_j,x'_k]$ for each cyclic permutation $(i,j,k)$ of $(1,2,3)$.

Suppose $\Phi:A \to B$ is an isomorphism of graded $\Bbbk$-algebras. The restriction of $\Phi$ to $A_1$ is a vector space isomorphism
$A_1 \to B_1$. It induces an isomorphism $\varphi:\PP(B_1^*) \to \PP(A_1^*)$. Let's denote the points
$(1,0,0,0),(0,1,0,0), (0,0,1,0), (0,0,0,1) \in \PP(B_1^*)$ by $e_0',e_1',e_2',e_3'$ respectively. 
Since $\Phi$ induces an isomorphism $A/([A,A]) \to B/([B,B])$, $\varphi$ restricts to an isomorphism
$\Proj(B/([B,B]) \to \Proj(A/([A,A])$. Therefore $\varphi(\{e_0',e_1',e_2',e_3'\})=\{e_0,e_1,e_2,e_3\}$.  Since each $x_m$ vanishes at exactly 3 points in $\{e_0,e_1,e_2,e_3\}$, $\Phi(x_m)$ vanishes at exactly 3 points in 
$\{e_0',e_1',e_2',e_3'\}$. It follows that there are non-zero scalars $\l_0,\l_1,\l_2,\l_3$ and a permutation $(p,q,r,s)$
of $(0,1,2,3)$ such that $\Phi(x_0)=\l_0x_p'$, $\Phi(x_1)=\l_1x_q'$, $\Phi(x_2)=\l_2x_r'$, and $\Phi(x_3)=\l_3x_s'$.

Since $\{x_0,x_i\}=[x_j,x_k]$ for every cyclic permutation $(i,j,k)$ of $(1,2,3)$,  
$$
\l_0\l_1\{x_p',x_q'\}=\l_2\l_3[x_r',x_s'], \quad \l_0\l_2\{x_p',x_r'\}=\l_3\l_1[x_s',x_q'], \quad \l_0\l_3\{x_p',x_s'\}=\l_1\l_2[x_q',x_r'].  
$$
Since $[x_0,x_i]=\a_i\{x_j,x_k\}$ for every cyclic permutation $(i,j,k)$ of $(1,2,3)$,  
$$
\l_0\l_1[x_p',x_q']=\a_1\l_2\l_3\{x_r',x_s'\}, \quad \l_0\l_2[x_p',x_r']=\a_2\l_3\l_1\{x_s',x_q'\}, \quad \l_0\l_3[x_p',x_s']=\a_3\l_1\l_2\{x_q',x_r'\}.  
$$

It follows that 
$$
[x_p',x_q']= \a_1\l_0^{-1}\l_1^{-1}\l_2\l_3\{x_r',x_s'\},
\qquad \qquad [x_r',x_s']=\l_0\l_1\l_2^{-1}\l_3^{-1}\{x_p',x_q'\},
$$
$$
[x_p',x_r']= \ \a_2\l_0^{-1}\l_2^{-1}\l_3\l_1\{x_s',x_q'\},
\qquad \qquad [x_s',x_q']=\l_0\l_2\l_3^{-1}\l_1^{-1}\{x_p',x_r'\}, 
$$
$$
[x_p',x_s']= \a_3\l_0^{-1}\l_3^{-1}\l_1\l_2\{x_q',x_r'\},
\qquad   \qquad [x_q',x_r']=\l_0\l_3\l_1^{-1}\l_2^{-1}\{x_p',x_s'\}.
$$
Therefore 
$\langle
p,q,r,s \rangle_B=( \a,\b,\c)= \langle
x_0,x_1,x_2,x_3\rangle$.  It now follows from \Cref{angle.invariants}
and the sentence after it that $\langle
0,1,2,3\rangle_B$ is a cyclic permutation of either $(
\a,\b,\c)$ or $(-\a,-\b,-\c)$; since
$\langle 0,1,2,3\rangle_B =(\a',\b',\c')$, the proof is complete.
\end{proof}

\section{The zero locus of the relations for $A(\a,\b,\c)$}
\label{sect.Gamma}

The ideas in \Cref{sect.4.6.algs} apply to all graded algebras defined by 4 generators and 6 quadratic relations, i.e., to all algebras $A$ of the form
$TV/(R)$ where $V$ and $R$ are as in the next paragraph. 

\subsection{Quadratic algebras on 4 generators with 6 relations}
\label{sect.4.6.algs}
Let  $V$ be a 4-dimensional vector space over $\Bbbk$, $R$ a 6-dimensional subspace of $V^{\otimes 2}$.
Let $\PP= \PP(V^*) \cong \PP^3$. Let $\G  \subseteq \PP \times \PP$ be the scheme-theoretic zero locus of $R$ (viewed as forms of bi-degree $(1,1)$). 
For example, if  $A$ is the polynomial ring, then $R$ consists of the skew-symmetric tensors and $\G$ is the diagonal.

Since $\dim_\Bbbk(R)=6=\dim(\PP^3 \times \PP^3)$, $\G \ne \varnothing$.

\begin{proposition}
\label{prop.20.pts.chow}
Suppose $\dim(\G)=0$. Then 
\begin{enumerate}
  \item 
  $\G$ consists of 20 points counted with multiplicity, and 
  \item 
the subspace of $V \otimes V$ that vanishes on  $\G$ is $R$ \cite[Thm. 4.1]{ShV02}. 
\end{enumerate}
\end{proposition}
\begin{proof}
(1) 
The Chow ring of $\PP^3$ is isomorphic to $\ZZ[t]/(t^4)$ with $t$ the class of a hyperplane. The Chow ring of 
 $\PP^3\times \PP^3$ is isomorphic to $\ZZ[s,t]/(s^4,t^4)$ and the class  of the zero locus of a non-zero element in 
 $V \otimes V$ is equal to $s+t$. If $\dim(\G)=0$, then the 
 class of $\G$ is $(s+t)^6$ since $\dim(R)=6$. But $(s+t)^6 = 20s^3t^3$ so the cardinality of $\G$ is 20 when its points
 are counted with multiplicity. 
 
 (2)
 This is \cite[Thm. 4.1]{ShV02}. 
\end{proof}

 \subsection{}
 We now explain our strategy for computing $\G$ for $A(\a,\b,\c)$.

Let $\bfx$ denote the row vector $(x_0,x_1,x_2,x_3)$ over
$\Bbbk\langle x_0,x_1,x_2,x_3\rangle$ and let $\bfx^\sT$ denote its
transpose.

The relations defining $A(\a,\b,\c)$ can be written as a single matrix
equation, $M\bfx^\sT=0$, over $\Bbbk\langle x_0,x_1,x_2,x_3\rangle$
where
 \begin{equation}
 \label{6-by-4.M}
 M \; :=\; 
 \begin{pmatrix}
 -x_1 & x_0 & -\a x_3 & -\a x_2 \\   -x_2 & -\b x_3 & x_0 & -\b x_1 \\   -x_3 & -\c x_2 & -\c x_1 & x_0 \\   -x_3 & -x_2 & x_1 & -x_0 \\   -x_1 & -x_0 & -x_3 & x_2 \\   -x_2 & x_3 & -x_0 & -x_1
\end{pmatrix}.
\end{equation}
The relations can also be written as $\bfx M'=0$ where
 \begin{equation}
 \label{4-by-6.M'}
 M' \; :=\; 
 \begin{pmatrix}
 -x_1     & -x_2    & -x_3     & -x_3 & -x_1  & -x_2 \\   
 x_0      & \b x_3 & \c  x_2 & x_2   &  -x_0 & -x_3 \\
 \a x_3  & x_0 &   \c x_1 & -x_1 & x_3 & -x_0 \\  
  \a x_2 & \b x_1 & x_0 & -x_0 & -x_2 & x_1
\end{pmatrix}.
\end{equation}

Consider the entries in $M$ (resp., $\bfx$) as linear forms on the left-hand (resp., right-hand) factor of
$\PP^3 \times \PP^3=\PP(V^*) \times \PP(V^*)$. Then $\G$ is the scheme-theoretic zero locus of the  six 
entries in $M\bfx^\sT$ when those entries
are viewed as bi-homogeneous elements in $\Bbbk[x_0,x_1,x_2,x_3] \otimes \Bbbk[x_0,x_1,x_2,x_3]$.

Let $\pr_1:\G \to \PP^3$ and $\pr_2:\G \to \PP^3$ be the projections  $\pr_1(p,p')=p$ and  $\pr_2(p,p')=p'$.

If $p \in \PP^3$, then $p \in \pr_1(\G)$ if and only if there is a point $p' \in \PP^3$ such that 
$M(p)\bfx^\sT(p')=0$; i.e., if and only if $\rank(M(p))<4$. Thus, $\pr_1(\G)$ is the scheme-theoretic zero locus of
the $4\times 4$ minors of $M$. Similarly, $\pr_2(\G)$ is the scheme-theoretic zero locus of
the $4\times 4$ minors of $M'$.

\begin{lemma}
\label{lem.8.points}
If $\l,\mu,\nu$ are non-zero scalars, then the intersection of the three quadrics
$$
x_0x_1-\l^2 x_2x_3=0, \qquad x_0x_2-\mu^2 x_1x_3=0, \qquad x_0x_3-\nu^2 x_1x_2=0, 
$$
consists of the eight points
\begin{align*}
& (0,1,0,0), \qquad (0,0,1,0), \qquad (\l\mu\nu,\l,\mu,\nu), \qquad \phantom{xxi} (\l\mu\nu,-\l,-\mu,\nu),
\\ 
& (1,0,0,0), \qquad (0,0,0,1),   \qquad (\l\mu\nu,-\l,\mu,-\nu), \qquad (\l\mu\nu,\l,-\mu,-\nu).
\end{align*}
\end{lemma}
\begin{proof}
The line $x_0-\l\mu x_3=x_1-\l\mu^{-1}x_2=0$ lies on the quadric $x_0x_1-\l^2 x_2x_3=0$ because
$$
x_0x_1-\l^2 x_2x_3 \; = \; (x_0-\l\mu x_3)x_1 + (x_1-\l\mu^{-1}x_2)\l\mu x_3
$$
and on the quadric $x_0x_2-\mu^2 x_1x_3=0$ because
$$
x_0x_2-\mu^2 x_1x_3 \; = \; (x_0-\l\mu x_3)x_2 - (x_1-\l\mu^{-1}x_2)\mu^2 x_3.
$$
Continuing in this vein, the lines $x_0=x_3=0$, $x_1=x_2=0$, $x_0-\l\mu x_3=x_1-\l\mu^{-1}x_2=0$, and 
$x_0+\l\mu x_3=x_1+\l\mu^{-1}x_2=0$, lie on the quadrics $x_0x_1-\l^2 x_2x_3=0$ and $x_0x_2-\mu^2 x_1x_3=0$. 
By B\'ezout's theorem, the intersection of these two quadrics is a curve of degree 4 in $\PP^3$ so is
the union of these four lines.

The quadric $x_0x_3-\nu^2 x_1x_2=0$ meets the line $x_0=x_3=0$  at $(0,1,0,0)$ and $(0,0,1,0)$;
 the  line $x_1=x_2=0$ at $(1,0,0,0)$ and $(0,0,0,1)$; the line $x_0-\l\mu x_3=x_1-\l\mu^{-1}x_2=0$ at $(\l\mu\nu,\l,\mu,\nu)$ and $(\l\mu\nu,-\l,-\mu,\nu)$; and the line $x_0+\l\mu x_3=x_1+\l\mu^{-1}x_2=0$ at $(\l\mu\nu,-\l,\mu,-\nu)$ and $(\l\mu\nu,\l,-\mu,-\nu)$.  The proof is complete.
\end{proof}

\begin{proposition}
\label{prop.Gamma.finite.iff}
The scheme $\G$ associated to the algebra $TV/(R) = A(\a,\b,\c)$ is finite if and only if $\a\b\c\ne 0$ and $\a+\b+\c+\a\b\c \ne 0$.
\end{proposition}
\begin{proof}
Before starting the proof we introduce some notation.

We label the following four polynomials in the symmetric algebra $SV$:
\begin{align*}
q & \; := \; x_0^2+ x_1^2+ x_2^2 + x_3^2, 
\\
q_1 & \;:= \; x_0^2-\b\c x_1^2-\c x_2^2 +\b x_3^2,  
\\
q_2 & \;:= \;  x_0^2+\c x_1^2-\a\c x_2^2 -\a x_3^2,  \qquad \text{and}
\\
q_3 & \;:= \; x_0^2-\b x_1^2+\a x_2^2 -\a\b x_3^2.
\end{align*}

We write $h_{ij}$ for the $4\times 4$  minor of $M$ obtained by deleting rows $i$ and $j$. 
Up to non-zero scalar multiples, 
\begin{align*}    
 h_{23} &  =  (x_0x_1-\a x_2x_3)q, \;   \; h_{46}   =   (x_0x_1+\a x_2x_3)q_1,    \;\;    h_{24}   =  (x_0x_1-x_2x_3)q_2,  \quad \phantom{x} h_{36}   =  (x_0x_1+x_2x_3)q_3,  
  \\       
   h_{13} &  =  (x_0x_2 -\b x_1x_3)q, \; \;  h_{14}   =  (x_0x_2+ x_1x_3)q_1,    \; \; \phantom{\a}  h_{45}   =  (x_0x_2+\b x_1x_3)q_2,     \quad    h_{35}   =  (x_0x_2-x_1x_3)q_3,    
  \\   
 h_{12} &  =  (x_0x_3-\c x_1x_2)q, \; \; h_{16}   =  (x_0x_3-x_1x_2)q_1,   \; \;  \phantom{\a}     h_{25}   =  (x_0x_3+x_1x_2)q_2, \quad  \phantom{x.} h_{56}   =  (x_0x_3+\c x_1x_2)q_3,    
  \\       
 h_{34} & \, = \,      
 (\a\b x_3^2-x_0^2)(x_1^2+x_2^2) + (\a x_2^2-\b x_1^2)(x_0^2+x_3^2)
 \\
& \, = \,     (\a\b x_3^2-x_0^2)q+(x_0^2+x_3^2)q_3 ,
 \\       
 h_{26} & \, = \,      
  (\a\c x_2^2-x_0^2)(x_1^2+x_3^2) + (\c x_1^2-\a x_3^2)(x_0^2+x_2^2)
   \\
& \, = \,     (\a\c x_2^2-x_0^2)q+(x_0^2+x_2^2)q_2,   
   \\       
 h_{15} & \, = \,      
  (\b\c x_1^2-x_0^2)(x_2^2+x_3^2) + ( \b x_3^2-\c x_2^2)(x_0^2+x_1^2)
     \\
& \, = \,     (\b\c x_1^2-x_0^2)q+(x_0^2+x_1^2)q_1.
 \end{align*}
 These are the ``same'' expressions as those in the proof of  \cite[Prop. 2.4]{SS92}.\footnote{The phrase ``making frequent use of (0.2.1)'' in  the second sentence in the proof of  \cite[Prop. 2.4]{SS92}
should be deleted in order to make that sentence true.}

We write  $g_{ij}$ for the $4\times 4$  minor of $M'$ obtained by deleting columns $i$ and $j$. 

If $p =(\l_0,\l_1,\l_2,\l_3) \in \PP^3$ we write $\ominus p$ for the point $(-\l_0,\l_1,\l_2,\l_3)$ and $M'(\ominus p)$ for the matrix $M'$ evaluated at $\ominus p$. The matrices $M'(\ominus p)$ and $-M(p)^\sT$ are almost the same: the only difference is that the top row of $M'(\ominus p)$ is the negative of the top row of $-M(p)^\sT$. This observation makes it easy to compute the 
 $4\times 4$  minors of $M'$ from the $4 \times 4$ minors of $M$. Doing that, 
up to non-zero scalar multiples we obtain
\begin{align*}    
 g_{23} &  =  (x_0x_1+\a x_2x_3)q, \;   \; g_{46}   =   (x_0x_1-\a x_2x_3)q_1,    \;\;    g_{24}   =  (x_0x_1+x_2x_3)q_2,  \quad \phantom{x} g_{36}   =  (x_0x_1-x_2x_3)q_3,  
  \\       
   g_{13} &  =  (x_0x_2 +\b x_1x_3)q, \; \;  g_{14}   =  (x_0x_2-x_1x_3)q_1,    \; \; \phantom{\a}  g_{45}   =  (x_0x_2-\b x_1x_3)q_2,     \quad    g_{35}   =  (x_0x_2+x_1x_3)q_3,    
  \\   
 g_{12} &  =  (x_0x_3+\c x_1x_2)q, \; \; g_{16}   =  (x_0x_3+x_1x_2)q_1,   \; \;  \phantom{\a}     g_{25}   =  (x_0x_3-x_1x_2)q_2, \quad  \phantom{x.} g_{56}   =  (x_0x_3-\c x_1x_2)q_3,    
  \\       
 g_{34} & \, = \,      
 (\a\b x_3^2-x_0^2)(x_1^2+x_2^2) + (\a x_2^2-\b x_1^2)(x_0^2+x_3^2)
 \\
& \, = \,     (\a\b x_3^2-x_0^2)q+(x_0^2+x_3^2)q_3 ,
 \\       
 g_{26} & \, = \,      
  (\a\c x_2^2-x_0^2)(x_1^2+x_3^2) + (\c x_1^2-\a x_3^2)(x_0^2+x_2^2)
   \\
& \, = \,     (\a\c x_2^2-x_0^2)q+(x_0^2+x_2^2)q_2,   
   \\       
 g_{15} & \, = \,      
  (\b\c x_1^2-x_0^2)(x_2^2+x_3^2) + ( \b x_3^2-\c x_2^2)(x_0^2+x_1^2)
     \\
& \, = \,     (\b\c x_1^2-x_0^2)q+(x_0^2+x_1^2)q_1.
 \end{align*}
In particular, up to non-zero scalar multiples, $g_{ij}(x_0,x_1,x_2,x_3)=h_{ij}(-x_0,x_1,x_2,x_3)$ for all $i$ and $j$. 
Hence $\pr_2(\G)=\ominus \pr_1(\G)$. 

It follows that $\pr_1(\G)$ is finite if and only if $\pr_2(\G)$ is finite if and only if $\G$ is finite.

($\Leftarrow$)
Suppose  $\a\b\c\ne 0$ and  $\a+\b+\c+\a\b\c \ne 0$.

Let $C$ be an irreducible component of $\pr_1(\G)$.
Since $h_{12}$, $h_{13}$, and $h_{23}$, vanish on $C$, either $q$ vanishes on $C$ or $C$ is in the zero locus of the other factors of $h_{12}$, $h_{13}$, and $h_{23}$; i.e., in the common zero locus of $x_0x_1-\a x_2 x_3$, 
$x_0x_2-\b x_1 x_3$, and $x_0x_3-\c x_1x_2$; but that common zero locus is finite by  \Cref{lem.8.points} so
either $C$ is finite or $q$ vanishes on $C$. Likewise, if $q_j \in \{q_1,q_2,q_3\}$, either $C$ is finite or $q_j$ vanishes on $C$.
Thus, either $C$ is finite or all four of $q$, $q_1$, $q_2$, and $q_3$, vanish on $C$.
However, the set $\{q,q_1,q_2,q_3\}$ is 
  linearly independent because the determinant 
\begin{equation*}
\det   \begin{pmatrix}
1 & 1 & 1 & 1 \\
    1&-\b\c&-\c&\b \\ 1&\c&-\a\c&-\a \\ 1&-\b &\a &-\a\b
  \end{pmatrix}  \;=\; - (\a+\b+\c+\a\b\c)^2 
\end{equation*}
is non-zero so the common zero-locus of $q$, $q_1$, $q_2$, and $q_3$, is empty.
We conclude that $C$ is finite. 
It follows that $\pr_1(\G)$, and hence $\G$, is finite.

($\Rightarrow$)
Suppose $\G$ is finite. 

If $\a+\b+\c+\a\b\c = 0$, then ${\rm span}\{q,q_1\} = {\rm span}\{q,q_2\} = {\rm span}\{q,q_3\}$.
It follows that all $h_{ij}$ vanish on $\{q=q_1=0\}$ whence $\{q=q_1=0\} \subseteq \pr_1(\G)$. 
But this is ridiculous because $\{q=q_1=0\}$ is a curve, hence infinite, so we conclude that
 $\a+\b+\c+\a\b\c \ne 0$.

If $\a=0$, then all $h_{ij}$ vanish on the line $x_0=x_1=0$; i.e., $\{x_0=x_1=0\} \subseteq \pr_1(\G)$; this is not the case because $\G$ is finite so we conclude that $\a \ne 0$.
If $\b=0$, then  $\{x_0=x_2=0\} \subseteq \pr_1(\G)$; this is not the case so we conclude that $\b \ne 0$. 
 If $\c=0$, then  $\{x_0=x_3=0\} \subseteq \pr_1(\G)$; this is not the case so we conclude that $\c \ne 0$. 
 Thus, $\a\b\c \ne 0$.
\end{proof}

\subsection{}
\label{sect.defn.P}

Suppose $\a\b\c\ne 0$. Let $\fP\subseteq \PP^3$ denote the set of points in the following table.
\begin{table}[htp]
\begin{center}
\begin{tabular}{|l||l|l|l|l|l|}
\hline
$\quad \fP_\infty$ & $\qquad \fP_0$ & $\qquad \fP_1$ & $\qquad \fP_2$ & $\qquad \fP_3$ & 
\\
\hline
\hline
$(1,0,0,0)$ & $(abc, a, b,c)$ &  $(a,-ia,-i,-1)$ & $(b,-1,-ib,-i)$   &   $(c,-i,-1,-ic)$  & 
\\
\hline
$(0,1,0,0)$ &$(abc, a, -b,-c)$   & $(a,-ia,i,1)$  &  $(b,-1,ib,i)$  &  $(c,-i,1,ic)$  & $\c_1$
\\
\hline
 $(0,0,1,0)$ &$(abc, -a, b,-c)$   & $(a,ia,-i,1)$  &  $(b,1,-ib,i)$ &  $(c,i,-1,ic)$   & $\c_2$
  \\
\hline
 $(0,0,0,1)$ & $(abc, -a, -b,c)$ & $(a,ia,i,-1)$ & $(b,1,ib,-i)$  &  $(c,i,1,-ic)$  & $\c_3$
  \\
\hline
\end{tabular}
\end{center}
\vskip .12in
\caption{The points in $\fP$}
\label{table.20.pts}
\end{table}
\newline

Let $\ominus:\PP(V^*) \to \PP(V^*)$ and $\theta:\fP \to \ominus \fP$ be the maps
$\ominus(\xi_0,\xi_1,\xi_2,\xi_3)=(-\xi_0,\xi_1,\xi_2,\xi_3)$ and  
$$
\theta(p) \; := \; 
\begin{cases}
p & \text{if $p \in \fP_\infty$,}
\\
\ominus p & \text{if $p \in  \fP_0$,}
\\
\ominus  \c_i(p) & \text{if $p \in \fP_i$, $i=1,2,3$.}
\end{cases}
$$

\begin{proposition}
\label{prop.Gamma}
Suppose the subscheme $\G\subseteq \PP^3 \times \PP^3$ determined by
the relations for $A(\a,\b,\c)$ is finite.  Then $\G$ is the graph of
the bijection $\theta:\fP \to \ominus \fP$ and
consists of 20 distinct points.
\end{proposition}
\begin{proof}
Since $\Gamma$ is finite both $\a\b\c$ and $\a+\b+\c+\a\b\c$ are non-zero. 
Since $\a\b\c \ne 0$, each column of \Cref{table.20.pts} consists of four distinct points.
It is easy to see that 
$$
\fP_\infty \cap\, (\fP_0  \cup\, \fP_1  \cup\, \fP_2  \cup \, \fP_3 ) \; = \;  \fP_1 \cap \, \fP_2  \; = \;  \fP_2 \cap \, \fP_3  \; = \;  \fP_3\cap\, \fP_1   \; = \; \varnothing.
$$
If $(a,\xi_1,\xi_2,\xi_3) \in \fP_0$, then $\xi_1=a\xi_2\xi_3$; if $(a,\xi_1,\xi_2,\xi_3) \in \fP_1$, then $\xi_1=-a\xi_2\xi_3$; hence $\fP_0\cap\, \fP_1 = \varnothing$. Similarly, if $(b,\xi_1,\xi_2,\xi_3) \in \fP_0$, then $\xi_2=b\xi_1\xi_3$ whereas if $(b,\xi_1,\xi_2,\xi_3) \in \fP_2$, 
then $\xi_2=-b\xi_1\xi_3$ so  $\fP_0\cap\, \fP_2 = \varnothing$. The same sort of argument shows that $\fP_0 \cap \, \fP_3 = \varnothing$. 
Thus, $\fP$ is the disjoint union of five sets each of which consists of four distinct points. Hence $\fP$ consists of 20 distinct points.

Let $\G_\theta$ denote the graph of $\theta:\fP \to \ominus \fP$. 

To complete the proof we must show that the vanishing locus in $\PP \times \PP$ of the polynomials
  \begin{equation}
 \label{6.polyns}
  \begin{cases}
  x_0\otimes x_1-x_1\otimes x_0   \, - \, \a(x_2\otimes x_3+x_3\otimes x_2) &  \qquad  x_0\otimes x_1+x_1\otimes x_0   \, - \, x_2\otimes x_3  \, + \, x_3\otimes x_2  \\
   x_0\otimes x_2-x_2\otimes x_0   \, - \, \b(x_3\otimes x_1+x_1\otimes x_3) & \qquad  x_0\otimes x_2+x_2\otimes x_0   \, - \, x_3\otimes x_1  \, + \, x_1\otimes x_3  \\
    x_0\otimes x_3-x_3\otimes x_0   \, - \, \c(x_1\otimes x_2+x_2\otimes x_1)  & \qquad  x_0\otimes x_3+x_3\otimes x_0   \, - \, x_1\otimes x_2  \, + \, x_2\otimes x_1
  \end{cases}
  \end{equation}
is exactly $\G_\theta$. 

Clearly, if $p \in \fP_\infty$, then all six polynomials vanish at $(p,p)=(p,\theta(p))$.

Suppose $p \in \fP_0$. Let $(i,j,k)$ be a cyclic permutation of $(1,2,3)$.
Since $(p,\theta(p))=(p,\ominus p)$, $x_0 \otimes x_i+ x_i \otimes x_0$ and $x_j \otimes x_k - x_k \otimes x_j$ vanish at 
$(p,\theta(p))$; the three polynomials in the second column of \Cref{6.polyns} therefore vanish at $(p,\theta(p))$. 
On the other hand,    $x_0\otimes x_i-x_i\otimes x_0   \, - \, \a_i(x_j\otimes x_k+x_k\otimes x_j) $ vanishes at 
$(p,\theta(p))$  if and  only if $2x_0\otimes x_i  \, - \, 2\a_ix_j\otimes x_k$ does. This vanishes at $(p,\ominus p)$ because
$2x_0x_i  \, - \, 2\a_ix_j x_k$ vanishes at $p$.

Let $(i,j,k)$ be a cyclic permutation of $(1,2,3)$.
Suppose $p =(\xi_0,\xi_1,\xi_2,\xi_3) \in \fP_i$. Then $\theta(p)=(\xi_0',\xi_1',\xi_2',\xi_3')$ where $\xi_i'=-\xi_i$ and $\xi_\ell'=\xi_\ell$ if $\ell \in \{0,1,2,3\}-\{i\}$. 
It follows that 
$$
\big(x_0\otimes x_i-x_i\otimes x_0   \, - \, \a_i(x_j\otimes x_k+x_k\otimes x_j)\big)\big\vert_{(p,\theta(p))}  
\;=\; - 2 \xi_0\xi_i -2\a_i\xi_j\xi_k,
$$
 $$
\big(x_0\otimes x_j +x_j\otimes x_0   \, - \,  x_k\otimes x_i + x_i\otimes x_k\big)\big\vert_{(p,\theta(p))}  
\;=\; 2 \xi_0\xi_j +2\xi_k\xi_i,
$$
$$
\big(x_0\otimes x_k +x_k\otimes x_0   \, - \,  x_i\otimes x_j + x_j\otimes x_i\big)\big\vert_{(p,\theta(p))}  
\;=\; 2 \xi_0\xi_k -2\xi_i\xi_j.
$$
A case-by-case inspection shows that these three expressions are zero;
thus, three of the polynomials in \Cref{6.polyns} vanish at
$(p,\theta(p))$. The other three polynomials in \Cref{6.polyns} also
vanish at $(p,\theta(p))$ because
$$
x_0 \otimes x_i + x_i \otimes x_0,  \quad  x_0 \otimes x_j - x_j \otimes x_0, \quad x_0 \otimes x_k - x_k \otimes x_0,  \phantom{andx}
$$
$$
x_i \otimes x_j + x_j \otimes x_i, \quad x_k \otimes x_i + x_i \otimes x_k,  \quad 
\text{and} \quad
x_j \otimes x_k - x_k \otimes x_j, 
$$
 vanish at $(p,\theta(p))$.    
   
 We have shown that the polynomials in \Cref{6.polyns} vanish on
 $\Gamma_\theta$. This completes the proof that
 $\G_\theta \subseteq \G$. In particular, $\fP \subseteq \pr_1(\G)$.
 To complete the proof of the proposition we must show the polynomials
 in \Cref{6.polyns} do not vanish outside $\G_\theta$ or,
 equivalently, that $\pr_1(\G)=\fP$.

With this goal in mind let $p \in \pr_1(\G)$. We observed in the proof of \Cref{prop.Gamma.finite.iff}, that 
$$
\{q=q_1=q_2=q_3=0\}  \; = \; \varnothing.
$$
If $q$ does not vanish at $p$, then  \Cref{lem.8.points} implies that $p \in \fP_\infty \cup \fP_0$.
Likewise, if $q_j \in \{q_1,q_2,q_3\}$ and does not vanish at $p$, then  \Cref{lem.8.points} 
tells us that $p \in \fP_\infty \cup \fP_j$. We conclude that $p \in \fP$. 
\end{proof}

\begin{corollary}
  If $\a\b\c \ne 0$ and $\a+\b+\c+\a\b\c\ne 0$, then $A(\a,\b,\c)$ is
  isomorphic to $TV/(R)$ where $R \subseteq V^{\otimes 2}$ is the
  subspace consisting of those $(1,1)$ forms that vanish on the graph
  of the bijection $\theta:\fP\to \ominus \fP$.
\end{corollary}

The deeper meaning of the data $(\fP,\theta)$ eludes us.

\subsection{Remarks}
\label{sect.remarks}
In \Cref{sect.remarks}  we  assume that $\a\b\c \ne 0$ but do not make any assumption on 
$\a+\b+\c+\a\b\c$.

\subsubsection{}
Let $\psi_1,\psi_2,\psi_3 \in \GL(V)$ be the maps defined in \Cref{ssect.psi_j}.
Let $\c_1,\c_2,\c_3 \in \GL(V)$ be the maps defined in \Cref{sect.gamma_js}.

Let $x_0^*,x_1^*,x_2^*,x_3^*$ be the dual basis to $x_0,x_1,x_2,x_3$. 
The contragredient action of the maps $\psi_j$ acting on $V^*$ is given by \Cref{contragred.autom}.
The subgroup of $\GL(V^*)$ generated by $\psi_1,\psi_2,\psi_3$ is isomorphic to $H_4$.
The center of $H_4$ acts trivially on $\PP(V^*)$  so we obtain an action of $\ZZ_4 \times \ZZ_4$ on $\PP(V^*)$.

It is easy to see that $\psi_j(\fP)=\fP$ for all $j$. We note that 
$$
\psi_j(abc,a,b,c) \; = \; \begin{cases}
(a,-ia,i,1) & \text{if $j=1$,} \\
(b,1,-ib,i)  & \text{if $j=2$,} \\
(c,i,1,-ic)  & \text{if $j=3$.}
\end{cases}
$$ 
It follows rather easily that $\fP_0 \cup \fP_1 \cup \fP_2 \cup \fP_3$ is a single orbit under the action of $H_4$ and 
therefore a single  $\ZZ_4 \times \ZZ_4$-orbit.  
 The subgroup $\{\id,\c_1,\c_2,\c_3\}$ of $\ZZ_4 \times \ZZ_4$ is an essential subgroup and, as is easy to see,  
 none of $\c_1,\c_2,\c_3$ fixes any point in $\fP_0 \cup \fP_1 \cup \fP_2 \cup \fP_3$ so the homomorphism 
 $$
\ZZ_4 \times \ZZ_4 \; \longrightarrow \; \{\text{permutations of } \fP_0 \cup \fP_1 \cup \fP_2 \cup \fP_3\}
$$
is injective. It follows that $\fP_0 \cup \fP_1 \cup \fP_2 \cup \fP_3$ consists of 16 distinct points. Hence $\fP$ consists of 20
distinct points (even without the hypothesis $\a+\b+\c+\a\b\c \ne 0$).

 \begin{table}[htp]
\begin{center}
\begin{tabular}{|c|c|c|c|c|c|c|c|}
\hline
 & $x_0^*$ &  $x^*_1$ &  $x^*_2$ &  $x^*_3$  
\\
\hline
$\psi_1 $  & $ix^*_1$ &   $(bc)^{-1}x^*_0$ &  $-c^{-1} x^*_3$ &  $ib^{-1} x^*_2$ $\phantom{\Big)}$
\\
\hline
$\psi_2$  & $i x^*_2$ &   $ic^{-1}x^*_3$ &  $(ac)^{-1}x^*_0$ &  $-a^{-1} x^*_1$ $\phantom{\Big)}$
\\
\hline
$\psi_3$  & $i x^*_3$ &   $-ib^{-1}x^*_2$ &  $ia^{-1}x^*_1$ &  $(ab)^{-1}  x^*_0$ $\phantom{\Big)}$
\\
\hline
\end{tabular}
\end{center} 
\caption{Contragredient action of $H_4$ on $V^*$.}
\label{contragred.autom}
\end{table}
   
\subsubsection{}
The points in $\fP_\infty$ are fixed by the action of $\ZZ_2 \times \ZZ_2$ given by \Cref{Gamma.action}.

\subsubsection{}
If $i\ne \infty$ and $p$ is the topmost point in the column $\fP_i$, then  the other points in that column are 
$\c_1(p)$, $\c_2(p)$, and $\c_3(p)$, in that order from top to bottom. Thus, when $j \ne \infty$, $\fP_j$ is a single $\ZZ_2 \times \ZZ_2$-orbit.

\subsection{The scheme $\G$ for the 4-dimensional Sklyanin algebras}
We review the Sklyanin algebra case (see \cite{SS92} and \cite{LS93}). 
Let $A= A(\a,\b,\c)$ be a non-degenerate Sklyanin algebra. 

Then $\G$ is the graph of an automorphism of 
$E \cup\{e_0,e_1,e_2,e_3\}$ where $E \subseteq \PP^3$  is the quartic elliptic curve  cut out by the equations
 \begin{equation}
 \label{eqns.for.E}
 \begin{cases}
  x_0^2+x_1^2+x_2^2+x_3^2 \; = \; 0,     
  \\
  x_0^2 - \b\c x_1^2 - \c x_2^2+\b x_3^2  \; = \; 0,
\\
x_0^2+\c x_1^2-\a\c x_2^2-\a x_3^2  \; = \; 0,
\\
x_0^2-\b x_1^2+\a x_2^2-\a\b x_3^2  \; = \; 0,
\end{cases}
\end{equation}
and $e_i$ is the vanishing locus of $\{x_0,x_1,x_2,x_3\}-\{x_i\}$. The
points $e_i$ are the vertices of the four singular quadrics that
contain $E$. The automorphism fixes each $e_i$ and its restriction to
$E$ is a translation automorphism.  Furthermore,
$R=\{f \in V \otimes V \; | \; f|_{\G}=0\}$. Thus, $R$ and $\G$
determine each other.



We fix a point $o \in E \cap \{x_0=0\}$ and impose a group law on $E$ such that $o$ is the identity and four points on $E$ are collinear if and only if their sum is $o$. The 2-torsion subgroup $E[2]$ is $E \cap \{x_0=0\}$.
We write $\oplus$ for the group law and $\ominus$ for subtraction, i.e., $p \oplus q=r$ if and only if $p=r \ominus q$.

If $p=(\xi_0,\xi_1,\xi_2,\xi_3) \in E$, then  $\ominus p=(-\xi_0,\xi_1,\xi_2,\xi_3)$.

See \cite[\S7]{CS15} for a longer explanation that uses the same notation as here.

\begin{proposition}
 If $A$ is a non-degenerate 4-dimensional Sklyanin algebra, then 
  \begin{enumerate}
  \item $\fP_0 \cup \fP_1 \cup \fP_2 \cup \fP_3\subseteq E$ where $E$
    is the elliptic curve given by the equations in \Cref{eqns.for.E};
  \item 
$\fP_0=\tau' \oplus E[2]$ where $\tau'=(abc,a,b,c)$;
  \item 
 $\fP_i=\ve_i \oplus \tau' \oplus E[2]$ and $E[2]=\{o,2\ve_1,2\ve_2,2\ve_3\}$ for suitable $\ve_1,\ve_2,\ve_3 \in E[4]$.
\end{enumerate} 
\end{proposition}

\section{Point schemes, graphs and flat families}\label{se.flat}

Consider the family of algebras $A(\a_1,\a_2,\a_3)$ as the parameters
$\a_i$ are allowed to vary. This section is devoted to studying the
behavior of the scheme $\G$ introduced in \Cref{sect.4.6.algs} as the
fiber of a family over the parameter space consisting of the points $(\a_1,\a_2,\a_3)$, or
in fact more generally over the family of six-dimensional relation
spaces for four-generator algebras.

\subsection{}
Throughout \Cref{se.flat}, $V$ denotes a fixed four-dimensional
space of generators for our quadratic algebras with a fixed basis
consisting of the generators $x_i$, $0\le i\le
3$, $\bG=\bG(6,V^{\otimes 2})$ is the Grassmannian of $6$-planes in
$V^{\otimes 2}$, and we regard the points of $\bG$ as spaces of
relations for four-generator-six-relator connected graded algebras, so
that $\bG$ will be the parameter space for the algebras in
question. We encode a relation space $R\in \bG$ as either a
$6\times 4$ matrix $M$ or a $4\times 6$ matrix $M'$ with entries in
$V$ analogous to \Cref{6-by-4.M} and respectively \Cref{4-by-6.M'}, so
that for $\bfx=(x_i)$ the relations $R$ read either $M\bfx^\sT=0$ or
$\bfx M'=0$.

As mentioned above, we denote by $\Gamma$ the family
$\pi:\Gamma\to \bG$ whose $R$-fiber $\Gamma_R$ for $R\in \bG$ is by
definition the subscheme of $\bP(V^*)\times\bP(V^*)$ consisting of the
pairs of points $(p,p')$ as in the discussion from
\Cref{sect.4.6.algs}, whose defining property is $M(p)\bfx^\sT(p')=0$.

We further set
\begin{equation*}
  \bX:=\pr_1(\Gamma),\quad \bX':=\pr_2(\Gamma),  
\end{equation*}
Let $U=\{R \in \bG\; | \; \dim(\bX_R)=0\}$ and define $U'$ similarly
for $\bX'$. Finally, given a family $\pi:\bS\to \bT$ and an open subset
$W\subseteq \bT$, we denote the restricted family $\pi^{-1}(W)$ by
$\bS_W$ by a slight abuse of notation.

When
the algebra $A_R$ corresponding to $R$ is Artin-Schelter regular and has some other good properties
that we will not specify here, the scheme
$\bX_R$ is (one incarnation of) the point scheme of $A_R$ (i.e. the
scheme of classes of point modules in $\QGr(A_R)$). Moreover, in many
cases of interest $\bX_R$ equals $\bX'_R$, and $\Gamma_R$ is the graph
of an automorphism of this scheme (e.g. for non-degenerate
$4$-dimensional Sklyanin algebras \cite{SS92}, for $3$-dimensional
AS-regular algebras \cite{ATV1}, etc.). Moreover, we saw above in
\Cref{prop.Gamma} that even when $\bX_R\ne \bX'_R$, the scheme
$\Gamma_R$ is often the graph of an isomorphism.

For these reasons, we regard $\Gamma$ and its projections $\bX$
and $\bX'$ as stand-ins for the point scheme even when we lack the
requisite regularity properties for this to be literally the case.

We first prove a statement analogous to \cite[Theorem 2.6]{CSV15}.
That result says, essentially, that the line schemes of connected
graded algebras with four generators and six quadratic relations form
a flat family provided they have minimal dimension. We prove here that
the family $\Gamma\to \bG$ is similarly well behaved.

First, we have the following analogue of \cite[Proposition
2.1]{CSV15}.

\begin{proposition}\label{pr.uu}
  The subset $U,U'\subset \bG$ are open and dense. 
\end{proposition}
\begin{proof}
  This is entirely analogous to the proof of \cite[Proposition
  2.1]{CSV15}. We focus on the case of $U$, to fix ideas.

First, Van den Bergh's result \cite{VdB88} that, generically,
four-generator-six-relator algebras have twenty point modules ensures
that $U$ contains a dense open subset of $\bG$. Let $\bX_i$ be the
irreducible components of $\bX$, and
$\pi_i:(\bX_i)_{\mathrm{red}}\to \pi(\bX_i)_{\mathrm{red}}$ the
restriction and corestriction of $\pi$ to $(\bX_i)_{\mathrm{red}}$.

Each $\pi_i$ is projective and hence closed. By \cite[Exercise
II.3.22(d)]{Hart} applied to each $\pi_i$ individually, the complement
of $U$, being the image of the closed subset
\begin{equation*}
  \{x\in \bX\ |\ x\text{ belongs to some component of }\bX_{\pi(x)}\text{ of dimension }\ge 1\}
\end{equation*}
of $\bX$, is closed in $\bG$.
\end{proof}

In conclusion, we get

\begin{corollary}\label{cor.w}
  The locus $W\subset \bG$ over which the family $\Gamma$ has
  zero-dimensional fibers is open and dense.
\end{corollary}
\begin{proof}
  The fiber $\Gamma_R$ is zero-dimensional if and only if its two
  projections $\bX_R$ and $\bX'_R$ are, so $W$ is simply the
  intersection $U\cap U'$. The conclusion now follows from \Cref{pr.uu}.
\end{proof}

We next turn, as hinted above, to proving certain regularity
properties for the families $\bX$, $\bX'$ and $\Gamma$ over the good
open loci of $\bG$ identified in \Cref{pr.uu,cor.w}.

\begin{lemma}\label{le.XCM}
  The schemes $\bX_U$ and $\bX'_{U'}$ are Cohen-Macaulay. 
\end{lemma}
\begin{proof}
  Once more, we focus on the case of $\bX_U$ without loss of
  generality.

Locally on $U$, the equations that define $\bX_U$ as a $U$-subscheme
of the relative projective space $\bP(V^*)_U=\bP(V^*)\times U$ are
given by the $4\times 4$ minors of a $6\times 4$ matrix. Moreover,
over $U$, $X_U$ has codimension $3$ in $\bP(V^*)_U$.

In general, the quotient by the ideal $I$ generated by the $r\times r$ minors of a $p\times q$ matrix in a Cohen-Macaulay ring is again Cohen-Macaulay, provided $I$ has maximal codimension $(p-r+1)(q-r+1)$ (see e.g. the discussion in \cite[$\S$ 18.5]{Ebud95} on determinantal rings and \cite{BV88} for a proof). In our case $p=6$, $q=4$ and $r=4$, hence the critical codimension is precisely $(6-4+1)(4-4+1)=3$. This concludes the proof. 
\end{proof}

\begin{theorem}\label{th.flat20}
  The families $\bX_U\to U$, $\bX'_{U'}\to U'$ and $\Gamma_W\to W$ are
  flat.
\end{theorem} 
\begin{proof}
We divide the argument into two parts. 

\vspace{.5cm}

{\bf (1): $\bX$ and $\bX'$.} Symmetry allows us to once again focus on
$\bX$. Given the Cohen-Macaulay property for $\bX_U$, the proof of the
theorem mimics that of \cite[Theorem 2.6]{CSV15} verbatim.

Let $x\in \bX_U$ be a point, and set $B=\cO_{\pi(x),U}$ and
$A=\cO_{x,\bX_U}$. In order to show that $A$ is flat as a $B$-module,
denote by $\fp$ the maximal ideal of $B$. Since the fiber
$(\bX_U)_{\pi(x)}$ is of minimal dimension $0$, we have the equality
\begin{equation*}
  \dim(A)=\dim(B)+\dim(A/A\fp).
\end{equation*}
This implies the flatness of $A$ over $B$ via \cite[Theorem 18.16
(b)]{Ebud95}, given that $A$ is Cohen-Macaulay by \Cref{le.XCM} and
$B$ is regular.

\vspace{.5cm}

{\bf (2): $\Gamma$.} The result for $\Gamma_W\to W$ follows from part
(1) and the observation that over $W$ the projection
$\pr_1:\Gamma\to \bX$ is an isomorphism.
\end{proof}

Finally, as an application of the flatness results just proven, we
provide an alternate argument for the fact that the 20 points in
\Cref{table.20.pts} exhaust the ``point scheme'' of
$A(\a_1,\a_2,\a_3)$ under certain non-degeneracy conditions on the
parameters $\a_i$. We begin with

\begin{corollary}\label{cor.flat20}
  For every $R\in U$, the scheme $\bX_R$ consists of 20 points
  counted with multiplicity. Similarly for $\bX'_{U'}$ and $\Gamma_W$.
\end{corollary}
\begin{proof}
  We can embed the family $\bX\to \bG$ into the relative projective
  space $\bP(V^*)_\bG= \bP(V^*)\times \bG$ in the usual fashion.

The flatness of \Cref{th.flat20} ensures that all fibers $\bX_R$ have
the same degree in $\bP(V^*)$ so long as $R\in U$, i.e.  when $\dim(\bX_R)=0$.
But we know there are
six-dimensional relation spaces $R$ where the degree is $20$ (e.g. the
algebras in \cite{VdB88,VVW98,SV99,CS15,ChV}). 

The case of $\bX'_{U'}$ is analogous, while that of $\Gamma_W$ follows
as in the proof of \Cref{th.flat20}, using the fact that $\pr_1:\Gamma\to \bX$
is an isomorphism over $W$.
\end{proof}

\begin{proposition}\label{pr.20}
  If $\sum \a_i+\prod \a_i\ne 0$ and $\prod\a_i\ne 0$ and   $A(\a_1,\a_2,\a_3)=TV/(R)$, then $\bX_R$ consists of the 20
  points in \Cref{table.20.pts}.
\end{proposition}
\begin{proof}
If we show that every closed point of $\bX_R$ is one of the  points in \Cref{table.20.pts}, then $\dim(\bX_R)=0$ so, by \Cref{cor.flat20} above, 
$\bX_R$ will consist of 20 points counted with multiplicity.  

Let $p$ be a  closed point in $\bX_R$. Consider the four
quadratic polynomials appearing as the right hand factors of $h_{12}$,
$h_{14}$, $h_{25}$, and $h_{36}$, which are
\begin{equation*}
  \sum x_i^2,\quad x_0^2-\b\c x_1^2-\c x_2^2 +\b x_3^2,\quad x_0^2+\c x_1^2-\a\c x_2^2 -\a x_3^2, \, \text{ and }  \, x_0^2-\b x_1^2+\a x_2^2 -\a\b x_3^2.  
\end{equation*}
Either all of them vanish at $p$, or at least one does not. 

Not all of them can vanish at $p$ because the determinant
\begin{equation*}
\det   \begin{pmatrix}
1 & 1 & 1 & 1 \\
    1&-\b\c&-\c&\b \\ 1&\c&-\a\c&-\a \\ 1&-\b &\a &-\a\b
  \end{pmatrix}  \;=\; - (\a+\b+\c+\a\b\c)^2 
\end{equation*}
is non-zero by hypothesis.

Thus, at least one of the four quadratic polynomials does not
  vanish at $p$. Note that the four quadratic polynomials are
permuted up to scaling by the $H_4$-action discussed above
(well-defined by our assumption $\alpha\beta\gamma\ne 0$), so we may
as well assume that $\sum x_i^2$ does not vanish at $p$. But then,
examining the minors $h_{12}$, $h_{13}$ and $h_{23}$, we see that $p$
is one of the finitely many points in
\begin{equation*}
  x_0x_3-\gamma x_1x_2 = x_0x_2-\beta x_1x_3 = x_0x_1-\alpha x_2x_3=0.
\end{equation*}
These points belong to $\fP$ so $p \in \fP$. However, it is easy to see that every $h_{ij}$ vanishes on $\fP$ so $\bX_R=\fP$. 
\end{proof}

%
%
%
%
%
%
%
%

 
\section{The algebras $R(a,b,c,d)$ of  Cho, Hong, and Lau}
\label{sect.CHL.algs}
 
 In this section $a,b,c$ do not denote square roots of $\a,\b,\c$.

\subsection{The definition}
\label{ssect.CHL.alg.defn}


Let  $(a,b,c,d) \in \Bbbk^4-\{0\}$ and define $R(a,b,c,d)$, or simply $R$, to be the free algebra $TV=\Bbbk\langle x_1,x_2,x_3,x_4\rangle$ modulo the relations
\begin{align*}
\hbox{(R1)} \qquad \phantom{xxxxx} & ax_4x_3+bx_3x_4 +cx_3x_2+dx_4x_1=0,   \\
\hbox{(R2)} \qquad \phantom{xxxxx}  & ax_3x_2+bx_2x_3 +cx_4x_3+dx_1x_2=0,   \\
\hbox{(R3)} \qquad \phantom{xxxxx}  & ax_2x_1+bx_1x_2 +cx_1x_4+dx_2x_3=0,   \\
\hbox{(R4)} \qquad \phantom{xxxxx} & ax_1x_4+bx_4x_1 +cx_2x_1+dx_3x_4=0,   \\
\hbox{(R5)} \qquad \phantom{xxxxx} &  ax_3x_1-ax_1x_3 +cx_4^2 -cx_2^2=0,      \\
\hbox{(R6)} \qquad \phantom{xxxxx} & bx_4x_2-bx_2x_4 +dx_3^2 -dx_1^2=0.   
\end{align*}
For example, $R(1,-1,0,0)$ is the commutative polynomial ring on 4 generators. 

Since $R(a,b,c,d)$ depends only on $(a,b,c,d)$ as a point in $\PP^3$, the family of algebras  
$R(a,b,c,d)$ is parametrized by $\PP^3$.
\Cref{prop.almost.Skly} concerns those algebras  $R(a,b,c,d)$ parametrized by the points on the quadric $ac+bd=0$. 
That quadric is isomorphic to $\PP^1 \times \PP^1$.

The lines 
\begin{align*}
  \ell_1& \; := \; \{a+id=c+ib=0\} \qquad \text{and} 
\\
  \ell_2& \; := \; \{a-id=c-ib=0\}
 \end{align*}
on that quadric and their open subsets
 \begin{align*}
  \ell_1^\circ & \; := \; \ell_1 \, - \,  \{ (0,i,1,0), \, (1,0,0,i), \, (1,-i,-1,i), \, (1,i,1,i),   \,(1,-1,i,i), \, (1,1,-i,i) \}             \qquad \text{and} 
\\
  \ell_2^\circ & \; := \; \ell_2 \, - \, \{(0,-i,1,0), \, (1,0,0,-i), \, (1,i,-1,-i), \, (1,-i,1,-i),   \, (1,-1,-i,-i), \, (1,1,i,-i) \}
 \end{align*}
play a distinguished role.
 
 \subsection{}
At \cite[Conj. 8.11]{CHL}, Cho, Hong, and Lau conjecture that when $ac+bd=0$, 
 $R$ is isomorphic to a 4-dimensional Sklyanin algebra,
i.e., isomorphic to $A(\a,\b,\c)$  for some $\a,\b,\c \in \Bbbk$ such that $\a+\b+\c+\a\b\c=0$. 
\Cref{prop.almost.Skly} shows that $R$ is isomorphic to a 4-dimensional Sklyanin algebra  if and only if $(a,b,c,d) \in \ell_1 \cup \ell_2$.
Nevertheless, $R$ is always isomorphic to $A(\a,\b,\c)$  for {\it some} $\a,\b,\c$.

\begin{proposition}
\label{prop.R.new.relns}
Let  $z_0=\frac{1}{2}(x_2+x_4)$, $z_1=\frac{1}{2}(x_1+x_3)$, $z_2=\frac{1}{2}(x_1-x_3)$, and $z_3=\frac{1}{2}(x_2-x_4)$.
The algebra $R(a,b,c,d)$ is equal to $\Bbbk\langle z_0,z_1,z_2,z_3\rangle$ modulo the relations
\begin{align*}
  (a-b-c+d)[z_0,z_1] &  \;=\; (-a-b+c+d)\{z_2,z_3\},    \\
 (-a+b+c+d)[z_0,z_2]  &  \;=\; (a+b-c+d)\{z_3,z_1\},    \\
 (a+b+c+d)\{z_0,z_1\}  &  \;=\; (a-b+c-d)[z_2,z_3],   \\
 (a+b+c-d)\{z_0,z_2\}    &  \;=\; (-a+b-c-d)[z_3,z_1], \\
 b[z_0,z_3]   &  \;=\; d\{z_1,z_2\}, \\
   c\{z_0,z_3\}  &  \;=\;a[z_1,z_2] .
\end{align*}
\end{proposition}
\begin{proof} 
Since $x_1=z_1+z_2$, $x_2=z_0+z_3$,  $x_3=z_1-z_2$,  and  $x_4=z_0-z_3$, 
the relations (R1)--(R4) can be replaced by the four relations
\begin{align*}
\hbox{$\frac{1}{2}$((R1)+(R3))}: \qquad & \; \;   (a+d)z_0z_1 +(b+c)z_1z_0+(a-d)z_3z_2+(b-c)z_2z_3=0 ,   \\
\hbox{$\frac{1}{2}$((R1)\,-\,(R3))}:\qquad & \; \;   (d-a)z_0z_2-(b+c)z_2z_0-(a+d)z_3z_1+(c-b)z_1z_3=0,   \\
\hbox{$\frac{1}{2}$((R4)+(R2))}: \qquad &\; \;   (b+c)z_0z_1 +(a+d)z_1z_0+(d-a)z_2z_3+(c-b)z_3z_2=0,   \\
\hbox{$\frac{1}{2}$((R4)\,-\,(R2))}: \qquad& \; \;   (a-d)z_2z_0+(b+c)z_0z_2-(a+d)z_1z_3+(c-b)z_3z_1=0.
\end{align*}

\begin{align*}
\hbox{$\frac{1}{2}$((R1)+(R3))}: \qquad & \; \;   (a+d)z_{01}+(b+c)z_{10}+(a-d)z_{32}+(b-c)z_{23}=0 ,   \\
\hbox{$\frac{1}{2}$((R1)\,-\,(R3))}:\qquad & \; \;   (d-a)z_{02}-(b+c)z_{20}-(a+d)z_{31}+(c-b)z_{13}=0,   \\
\hbox{$\frac{1}{2}$((R4)+(R2))}: \qquad &\; \;   (b+c)z_{01}+(a+d)z_{10}+(d-a)z_{23}+(c-b)z_{32}=0,   \\
\hbox{$\frac{1}{2}$((R4)\,-\,(R2))}: \qquad& \; \;   (a-d)z_{20}+(b+c)z_{02}-(a+d)z_{13}+(c-b)z_{31}=0.
\end{align*}
It follows that $R$ is equal to $\Bbbk\langle z_0,z_1,z_2,z_3 \rangle$ modulo the relations
\begin{align*}
\hbox{$\frac{1}{2}$((R1)+(R3)+(R2)+(R4))}: \qquad & \;  \;   (a+b+c+d)\{z_0,z_1\}   +(a-b+c-d)[z_3,z_2]=0,   \\
\hbox{$\frac{1}{2}$((R1)+(R3)\,-\,(R2)\,-\,(R4))}: \qquad & \;  \;    (a-b-c+d)[z_0,z_1]   +(a+b-c-d)\{z_3,z_2\}=0,    \\
\hbox{$\frac{1}{2}$((R1)\,-\,(R3)\,-\,(R2)+(R4))}:  \qquad & \;  \;  (-a+b+c+d)[z_0,z_2]   +(-a-b+c-d)\{z_1,z_3\}=0,   \\
\hbox{$\frac{1}{2}$((R1)\,-\,(R3)+(R2)\,-\,(R4))}: \qquad & \;  \;    (-a-b-c+d)\{z_0,z_2\}   +(a-b+c+d)[z_1,z_3]=0 \\
\hbox{$\frac{1}{2}$(R5)}:\qquad  & \; \;    a[z_1,z_2]   -c\{z_0,z_3\} =0,   \\
\hbox{$\frac{1}{2}$(R6)}:  \qquad& \; \;    b[z_0,z_3]   -d\{z_1,z_2\}=0.
\end{align*}
Rearranging these gives the presentation in the statement of the proposition.
\end{proof}


\begin{proposition}
\label{prop.almost.Skly}
Let  $a,b,c,d \in \Bbbk$ and define  $p:=a+b$, $q:=a-b$, $r:=c+d$, and $s:=c-d$. 
Suppose  that $ac+bd=0$ and 
 \begin{equation}
\label{bad.abcd}
abcd (p+r)(p-r)(p+s)(p-s)(q+r)(q-r)(q+s)(q-s) \; \ne \; 0.
\end{equation}
\begin{enumerate}
  \item 
$R(a,b,c,d) \cong A(\a,\b,\c)$ where 
$$
\a = \frac{r^2-p^2}{q^2-s^2}, \qquad \b = \frac{p^2-s^2}{r^2-q^2}, \quad \hbox{and} \quad 
\c  =\frac{cd}{ab}\, . 
$$ 
  \item 
   $\a+\b+\c+\a\b\c=0$ if and only if $(a,b,c,d) \in \ell_1 \cup \ell_2$.  
\item{}
If  $(a,b,c,d) \in \ell_1 \cup \ell_2$, then $R(a,b,c,d)$ is isomorphic to the Sklyanin algebra $A(\a,1,-1)$ where 
$$
\a = \begin{cases}
 (b+d+ib-id)^2 (b+d-ib+id)^{-2} & \text{if $a+id=c+ib=0$,}
 \\
(b+d-ib+id)^2(b+d+ib-id)^{-2} & \text{if $a-id=c-ib=0$,}
\end{cases}
$$
and is generated by homogeneous degree-one elements 
$Y_{\pm},K,K'$ such that 
$$
KY_{\pm}=\mp iY_{\pm}K, \quad  K'Y_{\pm}=\pm iY_{\pm}K', \quad [Y_+,Y_-]=i(K'^2-K^2), \quad \text{and} \quad [K,K']=i\a (Y_+^2-Y_-^2).
$$
\end{enumerate}
 \end{proposition}
\begin{proof}
(1)
Condition \Cref{bad.abcd} ensures that the denominators in the  expressions for $\a,\b,\c,$ are  non-zero.

Let $z_0,z_1,z_2,z_3$ be as in \Cref{prop.R.new.relns}.  Condition
\Cref{bad.abcd} ensures that the denominators in the following
expressions are non-zero, so $R$ is defined by the following
relations:
\begin{align*} 
[z_0,z_1]   & \; = \;    \frac{r-p}{q-s} \, \{z_2,z_3\},   \qquad  \{z_0,z_1\}    \; = \;  \frac{q+s}{p+r}\, [z_2,z_3], \\
[z_0,z_2]   & \; = \;  \frac{p-s}{r-q} \,\{z_3,z_1\},   \qquad  \{z_0,z_2\}   \; = \;   \frac{q+r}{p+s}  \, [z_3,z_1] ,  \\
[z_0,z_3]  & \; = \;  \frac{d}{b}  \, \{z_1,z_2\},  \qquad   \! \qquad  \{z_0,z_3\}  \; = \;    \frac{a}{c} \, [z_1,z_2]  . 
\end{align*} 
For brevity, let's write these relations as 
\begin{align*} 
 [z_0,z_1]   & \; = \;   \mu_1\, \{z_2,z_3\},   \qquad  \{z_0,z_1\}    \; = \; \nu_1\, [z_2,z_3], \\
 [z_0,z_2]   & \; = \;  \mu_2\,\{z_3,z_1\},   \qquad  \{z_0,z_2\}   \; = \;  \nu_2  \, [z_3,z_1],  \\
[z_0,z_3]  & \; = \; \mu_3 \, \{z_1,z_2\},  \qquad \, \{z_0,z_3\}  \; = \;  \nu_3 \, [z_1,z_2].
\end{align*} 
Define $v_0:=z_0$, $v_1:=\sqrt{\nu_2\nu_3}\, z_1$, $v_2:=\sqrt{\nu_3\nu_1}\,z_2$, and $v_3:=\sqrt{\nu_1\nu_2}\, z_3$. 
Thus, $R$ is the free algebra $\Bbbk\langle v_0,v_1,v_2,v_3 \rangle$ modulo the relations 
\begin{align*} 
 [v_0,v_1]   & \; = \;   \a \, \{v_2,v_3\},   \qquad  \{v_0,v_1\}    \; = \;  \, [v_2,v_3], \\
 [v_0,v_2]   & \; = \;  \b\,\{v_3,v_1\},   \qquad  \{v_0,v_2\}   \; = \;    \, [v_3,v_1] ,  \\
[v_0,v_3]  & \; = \; \c \, \{v_1,v_2\},  \qquad \, \{v_0,v_3\}  \; = \;  \, [v_1,v_2], 
\end{align*} 
where 
$$
\a=\mu_1\nu_1^{-1} =\frac{r^2-p^2}{q^2-s^2}, \qquad \b=\mu_2\nu_2^{-1} =\frac{p^2-s^2}{r^2-q^2}, \quad \hbox{and} \quad \c =\mu_3\nu_3^{-1} =\frac{cd}{ab}.
$$

(2)
The expression
$
(q^2-s^2)(r^2-q^2)ab\big( \a+\b+\c+\a\b\c)
$ 
is equal to
\begin{align*}
& (r^2-p^2)(r^2-q^2)ab \; + \; (p^2-s^2)(q^2-s^2)ab\; + \;  (q^2-s^2)(r^2-q^2)cd \; + \; (r^2-p^2)(p^2-s^2)cd 
\\
& \;=\; ab \big(r^4+s^4+2p^2q^2 \,- \, (p^2+q^2)(r^2+s^2)\big) \;-\; cd\big(p^4+q^4+2r^2s^2 \, - \,(p^2+q^2)(r^2+s^2)\big)
\\
& \;=\;  ab \big(2p^2q^2 - 2r^2s^2+ (r^2+s^2-p^2-q^2)(r^2+s^2)\big)
\\
&\phantom{xxxxxxxxxxx} \,-\, cd\big(2r^2s^2-2p^2q^2 -(r^2+s^2-p^2-q^2)(p^2+q^2)\big)
\\
& \;=\; 2(ab+cd)(p^2q^2 - r^2s^2)+ (r^2+s^2-p^2-q^2)(abr^2+abs^2+cdp^2+cdq^2).
\end{align*}
But $p^2+q^2=2(a^2+b^2)$ and $r^2+s^2=2(c^2+d^2)$ so
$$
abr^2+abs^2+cdp^2+cdq^2\; = \; 2(ac+bd)(ad+bc) \; = \; 0.
$$ 
Therefore
$
(q^2-s^2)(r^2-q^2)ab\big( \a+\b+\c+\a\b\c) = 2(ab+cd)(p^2q^2 - r^2s^2)$.  
Hence $\a+\b+\c+\a\b\c=0$ if and only if $(ab+cd)(p^2q^2 - r^2s^2)=0$; i.e., if and only if
$(a^2-b^2 - c^2+d^2)(a^2-b^2+c^2-d^2)(ab+cd) =0$.

The locus $ac+bd=(a^2-b^2 - c^2+d^2)(a^2-b^2+c^2-d^2)(ab+cd) =0$ consists of 8 lines:  
the locus  $ac+bd=a^2-b^2-c^2+d^2=0$ is the union of the four lines
 $$
a-b=c+d=0, \quad  a+b=c-d=0, \quad a+id=c+ib=0, \quad a-id=c-ib=0;
$$ 
the locus  $ac+bd=a^2-b^2+c^2-d^2=0$ is the union of the four lines
$$
a-b=c+d=0, \quad  a+b=c-d=0, \quad a+d=b-c=0, \quad a-d=b+c=0;
$$
the locus  $  ac+bd=ab+cd=0$ is the union of the four lines
$$
a=d=0, \qquad b=c=0, \qquad  a+d=b-c=0, \qquad a-d=b+c=0.
$$ 
If $(a,b,c,d)$ is on the line $a-b=c+d=0$ or on the line $a-d=b+c=0$, then $0=a-b-c-d=q-r$; hypothesis \Cref{bad.abcd} excludes this possibility so 
$(a,b,c,d)$ is not on either of those lines.  
If $(a,b,c,d)$ is on the line $a=d=0$ or on the line $b=c=0$, then $abcd=0$; hypothesis \Cref{bad.abcd} excludes this possibility so 
$(a,b,c,d)$ is not on either of those lines. 
If $(a,b,c,d)$ is on the line $a+d=b-c=0$ or on the line $a+b=c-d=0$, then $0=a+b-c+d=p-s$; hypothesis \Cref{bad.abcd} excludes this possibility so 
$(a,b,c,d)$ is not on either of those lines. 
Thus,  $\a+\b+\c+\a\b\c=0$ if and only if $(a,b,c,d)$ is on the  union of the lines $a+id=c+ib=0$ and $a-id=c-ib=0$; i.e., 
$(a,b,c,d)\in \ell_1 \cup \ell_2$.

Not every $(a,b,c,d)\in \ell_1 \cup \ell_2$ satisfies \Cref{bad.abcd}. 
The points $(a,b,c,d)$ on the line $a+id=c+ib=0$ that do not satisfy \Cref{bad.abcd} are
\begin{equation}
\label{eq.bad.pts.1}
(0,i,1,0), \quad (1,0,0,i), \quad (1,-i,-1,i), \quad (1,i,1,i),   \quad (1,-1,i,i), \quad (1,1,-i,i) .
\end{equation}
The points $(a,b,c,d)$ on the line $a-id=c-ib=0$ that do not satisfy \Cref{bad.abcd} are
\begin{equation}
\label{eq.bad.pts.2}
(0,-i,1,0), \quad (1,0,0,-i), \quad (1,i,-1,-i), \quad (1,-i,1,-i),   \quad (1,-1,-i,-i), \quad (1,1,i,-i) .
\end{equation}

(3) Suppose $a+id=c+ib=0$. Then $(\a,\b,\c)=(-1,\b,1)$ where
$\b=(b+d+ib-id)^2/ (b+d-ib+id)^2$.  As $(a,b,c,d)$ runs over the
points in \Cref{eq.bad.pts.1}, $\b$ takes on the values $-1$, $-1$,
$1$, $1$, $0$, $\infty$, respectively.  As $(a,b,c,d)$ runs over the
other points on the line $a+id=c+ib=0$, $\b$ takes on every value in
$\Bbbk-\{0,\pm 1\}$.

Suppose $a-id=c-ib=0$. Then $(\a,\b,\c)=(-1,\b,1)$ where
$\b=(b+d-ib+id)^2/ (b+d+ib-id)^2$.  As $(a,b,c,d)$ runs over the
points in \Cref{eq.bad.pts.2}, $\b$ takes on the values $-1$, $-1$,
$1$, $1$, $\infty$, $0$.  As $(a,b,c,d)$ runs over the other points on
the line $a-id=c-ib=0$, $\b$ takes on every value in
$\Bbbk-\{0,\pm 1\}$.
 
By \Cref{lem.isom}, $A(\b,1,-1) \cong A(-1,\b,1) \cong A(1,-1,\b)$ so to prove (3) it suffices to show that $A(\a,1,-1)$ has generators 
$Y_{\pm},K,K'$ satisfying the stated relations. We do this in \Cref{prop.special.case} below.
\end{proof}

\begin{proposition}
\label{prop.special.case}
Suppose $\a\in \Bbbk - \{0,\pm 1\}$. Let $i$ be a square root of $-1$. 
\begin{enumerate}
  \item 
$A(\a,1,-1)$ is a Sklyanin algebra $A(E,\tau)$ with $\tau$ being translation by a 4-torsion point.
  \item 
   There is a basis $Y_{\pm},K,K'$ for $A(\a,1,-1)_1$ such that
$$
KY_{\pm}=\mp iY_{\pm}K, \qquad K'Y_{\pm}=\pm iY_{\pm}K',
$$
$$
[Y_+,Y_-]=i(K'^2-K^2), \qquad [K,K']=i\a (Y_+^2-Y_-^2).
$$
\end{enumerate}
\end{proposition}
\begin{proof}
Let $A=A(\a,1,-1)$. By \Cref{lem.isom}, $A(\a,1,-1) \cong A(1,-1,\a)  \cong A(-1,\a,1)$. 
Algebras of the form  $A(1,-1,\a)$ are those identified in equation (1.9.4) of \cite{SS92} so the results in \cite{SS92} apply to $A$. 
The  zero locus of the $4 \times 4$ minors in  the proof of \cite[Prop. 2.4]{SS92} is the curve $E$ given by the equations
\begin{equation}
  \label{eq:ealpha}
x_0^2+x_1^2+ x_2^2+x_3^2 \; = \; x_0^2-x_1^2+ \a x_2^2-\a x_3^2 \; = \; 0.  
\end{equation}
The restrictions on $\a$ imply that the Jacobian matrix has rank 2 at all points of $E$ so $E$ { is} an elliptic curve. 
(The description of $E$, in particular, the formula for the polynomial $g_2$, in \cite[Prop. 2.4]{SS92} does not make sense when $\b=1$.)
The formula for the automorphism $\s:E \to E$ in \cite[Cor. 2.8]{SS92} is
\begin{equation}
  \label{eq:auto}
\s\begin{pmatrix} x_0 \\ x_1 \\ x_2 \\ x_3 \end{pmatrix} \; = \;
\begin{pmatrix}
\phantom{-ii}2\a x_1x_2x_3 - x_0 (-x_0^2-x_1^2- \a x_2^2+\a x_3^2)
\\
\phantom{ii}2\a x_0x_2x_3 + x_1 (\phantom{-i}x_0^2+x_1^2- \a x_2^2+\a x_3^2)
\\
\phantom{ii}2 x_0x_1x_3 \phantom{-}+ x_2 (\phantom{-i}x_0^2-x_1^2+ \a x_2^2+\a x_3^2)
\\
-2x_0 x_1x_2 \phantom{iii}+x_3 (\phantom{-i}x_0^2-x_1^2- \a x_2^2-\a x_3^2)
\end{pmatrix}
\; = \;
\begin{pmatrix} \phantom{-}  x_1 \\ \phantom{-}   x_0 \\ \phantom{-}  x_3 \\ -x_2 \end{pmatrix}   
\end{equation}
where the last equality uses the  fact that $x_0^2-x_1^2+ \a x_2^2-\a x_3^2=0$ on $E$. The formula for $\s$ can also be verified by observing that 
$$ 
 \label{6-by-4.M}
 \begin{pmatrix}
 -x_1 & x_0 & -\a x_3 & -\a x_2 \\   -x_2 & - x_3 & x_0 & - x_1 \\   -x_3 &  x_2 &  x_1 & x_0 \\   -x_3 & -x_2 & x_1 & -x_0 \\   -x_1 & -x_0 & -x_3 & x_2 \\   -x_2 & x_3 & -x_0 & -x_1
\end{pmatrix}
\begin{pmatrix} \phantom{-}  x_1 \\ \phantom{-}   x_0 \\ \phantom{-}  x_3 \\ -x_2 \end{pmatrix}  
\; = \; 0
$$
for all $(x_0,x_1,x_2,x_3) \in E$;  the $6 \times 4$ matrix in the previous equation is  the $6 \times 4$ matrix in \Cref{6-by-4.M}

Corollary 2.11 in \cite{CS17} involves elements $a,b,c \in \Bbbk$ such that $a^2=\a$, $b^2=1$, and $c^2=-1$; let $(a,b,c)=(a,1,i)$; 
 \cite[Cor. 2.11]{CS17} then says there is a 4-torsion point  $\ve_1 \in E$ such that if $p=(x_0,x_1,x_2,x_3) \in E$, then
$$
p+\ve_1 \; = \;  (i x_1 , -i x_0, i x_3, i x_2)  \; = \;  (x_1 , -x_0, x_3, x_2);
$$
by \cite[\S2.6]{CS17}, there is a 2-torsion point $\c_2 \in E$ such that
$
p+\ve_1 + \c_2    \; = \;  (x_1 , x_0, x_3, -x_2)
$;
thus
$$
\s(p) \; = \;   p+\ve_1 + \c_2.
$$

Calculations like those in \cite[\S1.2]{SS92} show that  $Y_{\pm}:=x_0 \pm x_1$, $K:=x_0+x_1$, $K':=x_0-x_1$, satisfy the relations in (2). 
\end{proof}

Part (2) of \Cref{prop.special.case} remains true when $A=A(0,1,-1)$
and in that case, $A(0,1,-1)$ is a homogenization of the quantized
enveloping algebra $U_q(\fsl_2)$ with $q=-i$. See \cite[\S2.4]{CSW16}
for details.

\subsection{Parameter spaces and modular curves}\label{subse.mod}

After \Cref{prop.special.case}(1),  it is natural to ask which pairs $(E,\tau)$ have the property that $A(E,\tau) \cong R(a,b,c,d)$ for some point $(a,b,c,d)$ 
in the ``Sklyanin locus'' $\ell_1^\circ \cup \ell_2^\circ$.
Similarly, we can ask how much redundancy there is in this parametrization: how many $(a,b,c,d) \in \ell_1^\circ \cup \ell_2^\circ$ lead to the same pair $(E,\tau)$?

Since the transformation
\begin{equation*}
  a\longleftrightarrow -c,  \qquad  b\longleftrightarrow -d
\end{equation*}
interchanges  $\ell_1$ and $\ell_2$ and intertwines the respective transformations
\begin{equation*}
  (a,b,c,d)\mapsto \b,
\end{equation*}
it suffices to consider what happens for $\ell_1$.

Note first that the map 
\begin{equation}\label{eq:squares}
  \varphi:(a,b,c,d)\mapsto \frac{(b+d+ib-id)^2}{(b+d-ib+id)^2}
\end{equation}
recovering the parameter $\alpha$ of the Sklyanin algebra $A(\alpha,1,-1)$ from $(a,b,c,d)\in \ell_1^\circ$
is a two-fold cover of
\begin{equation*}
X \; := \;   \bP^1- \{\pm 1,0,\infty\}. 
\end{equation*}
Now, to each $\alpha\in X$ associate the elliptic curve $E_\alpha$ of point modules of $A(\alpha,1,-1)$, defined by \Cref{eq:ealpha}. Since furthermore the point $(1,1,i,i)$ belongs to all $E_\alpha$, the $\alpha$-indexed family $E\to X$ of elliptic curves $E_\alpha$ over $X$ has a section. 

Finally, \Cref{eq:auto} defines an automorphism of order $4$ of the
family $E\to X$. Since the section $(1,1,i,i)$ puts on $E$ a unique
structure of an abelian curve over $X$ \cite[Theorem 2.1.2]{KM}, we
can identify said automorphism with a point of $E$ of order
(precisely) $4$. In other words, we obtain a family of abelian curves
over $X$ with marked order-$4$ points. This moduli problem is
represented by the modular curve $Y_1(4)$ classifying such data (see
e.g. \cite[Theorem 8.2.1]{DI} and references therein), and hence we
obtain a morphism
\begin{equation*}
  \psi:X\to Y_1(4). 
\end{equation*}
The following results give the full picture of the parametrization of the Cho-Hong-Lau algebras.

\begin{proposition}\label{pr.modular}
  The map $\psi:X=\bP^1- \{\pm 1,0,\infty\} \longrightarrow Y_1(4)$
  defined above as
  \begin{equation*}
   X\ni\alpha\mapsto (E,\tau)\in Y_1(4)\text{ for }A(\alpha,1,-1)\cong A(E,\tau)
  \end{equation*}
  is a two-fold cover, identifying $\pm\alpha$.
\end{proposition}
\begin{proof}
  Note first that the automorphism
  \begin{equation*}
    (x_0,x_1,x_2,x_3)\mapsto (x_0,x_1,x_3,-x_2) 
  \end{equation*}
of $\bP^3$ interchanges the elliptic curves $E_\alpha$ and $E_{-\alpha}$, and moreover it intertwines their respective order-$4$ automorphisms defining them as points of $Y_1(4)$. This implies that $\psi$ factors through a morphism
\begin{equation*}
  \psi':X/\pm \to Y_1(4). 
\end{equation*}
Since $Y_1(4)$ is known to have three cusps and the left-hand side is a thrice-punctured projective line, $\psi'$ extends to a endomorphism $\overline{\psi'}$ of $\bP^1$. It follows from the fact that three distinct points have singleton preimages that $\overline{\psi'}$ is an isomorphism, and hence so is $\psi'$. 
\end{proof}

In conclusion, we have

\begin{corollary}\label{cor.modular}
The maps $\ell^\circ_i \longrightarrow Y_1(4)$, $i=1,2$, that send $(a,b,c,d)$ to 
the underlying elliptic curve and automorphism of the Sklyanin algebra $R(a,b,c,d)$ are fourfold covers. 
\end{corollary}
\begin{proof}
  Simply compose $\psi$ and its analogue for $\ell_2$ (which are double covers by \Cref{pr.modular}) with the two-fold covers of the form \Cref{eq:squares}. 
\end{proof}

\section{Central elements}
\label{sect.centers}

\subsection{Central elements in $A(\a,\b,\c)$}

The next result  is often asserted but we could not find a proof in the literature so we include one here. 

\begin{proposition}
\label{prop.Skly.center}
Let $\Bbbk$ be any field. If $\{\a,\b,\c\} \cap \{0 , \pm 1\}=\varnothing$ and   $\a+\b+\c+\a\b\c = 0$, then $-x_0^2+x_1^2+x_2^2+x_3^2$ and $x_0^2+\b\c x_1^2-\c x_2^2+\b x_3^2$ belong to the center of $A(\a,\b,\c)$.
\end{proposition}
\begin{proof}
Let's simplify the notation by omitting the $x$'s and just retaining the subscripts, so $kji$ denotes
$x_kx_jx_i$, $ii0$ denotes $x_ix_ix_0$, and so on. We also write $\{i,j\}$ for $\{x_i,x_j\}$,  $[i,j]$ for $[x_i,x_j]$, etc. 

For each cyclic permutation $(i,j,k)$ of $(1,2,3)$, we define
$$
c_i   \;  :=\;  [x_0,x_i]- \a_i \{x_j,x_k\}  \qquad \hbox{and} \qquad
 a_i  \;  := \; \{x_0,x_i\}-[x_j,x_k].
 $$
Straightforward computations in the free algebra $\Bbbk\langle x_0,x_1,x_2,x_3\rangle$ show that 
\begin{align*}
\{x_i,c_i\} & \; = \;  \{i, [0,i]\}- \a_i \{i, \{j,k\}\}  \notag
\\
& \; =\;    [0,ii]- \a_i (ijk+kji)- \a_i (ikj+jki), \qquad \hbox{and}
\\
[x_i,a_i] & \;=\;  [i,\{0,i\}]-[i, [j,k]]   \notag
\\
& \;=\;    [ii,0]- (ijk+kji)+(ikj+jki).
\end{align*}
When $\a_1+\a_2+\a_3+\a_1\a_2\a_3 = 0$, error-prone calculations\footnote{In carrying out these calculations one should not attempt to ``simplify'' 
the expressions $ijk+kji$ and $ikj+jki$. }
show that  
\begin{align*}
&(1+\a_ 2\a_ 3)  \{ x_1,c_1\} \;\; + \; \;  \a_ 2 \a_ 3[x_1,a_1]  \; \; + \; \;  (1+\a_ 3)\{ x_2,c_2\}  \; +\; \a_3[x_2,a_2] 
\\
 &  \phantom{xxxxxxxxx1ixxxxxxxxxixxxxxxx}  \; \; + \; \; 
 (1-\a_ 2)\{ x_3,c_3\} \; - \;\a_2 [x_3,a_3] 
 \end{align*}
 equals  $[x_0,x_1^2+x_2^2+x_3^2]$.  
 Hence $ [x_0,-x_0^2+x_1^2+x_2^2+x_3^2] =0$ in $A(\a_1,\a_2.a_3)$. 
 
 A similar calculation shows that
  \begin{align*}
& (1+\a_2)\big(\{x_0,c_1\}  +\a_1[x_0,a_1] \big)     
  \; \; + \; \; (1-\a_1)[x_3,c_2] + (1+\a_1+2\a_1\a_2)\{x_3,a_2\} 
\\
 &  \phantom{xxxxxxxxxiixxxxxxxxxxxxxxxxxxxx} 
  \; \; - \; \;  (1+\a_1\a_2)\big([x_2,c_3] +\{x_2,a_3\} \big) 
 \\
&  \phantom{xxx} \; \;=\; \; (1+\a_1)(1+\a_2) [x_1,-x_0^2+x_1^2+x_2^2+x_3^2].
 \end{align*}
 Hence $ [x_1,-x_0^2+x_1^2+x_2^2+x_3^2]=0$. The transformation $x_0 \mapsto x_0$, $x_i \mapsto x_{i+1}$ for $i=1,2,3$, and $\a_i \mapsto \a_{i+1}$ for $i=1,2,3$, 
 leaves $-x_0^2+x_1^2+x_2^2+x_3^2$ fixed; it follows that $ [x_2,-x_0^2+x_1^2+x_2^2+x_3^2]=0$ and then that $ [x_3,-x_0^2+x_1^2+x_2^2+x_3^2]=0$.
 This completes the proof that $-x_0^2 +x_1^2+x_2^2+x_3^2$ belongs to the center of $A(\a_1,\a_2,\a_3)$ when 
 $\{\a_1,\a_2,\a_3\} \cap \{0 , \pm 1\}=\varnothing$ and  $\a_1+\a_2+\a_3+\a_1\a_2\a_3 = 0$. 
  
The automorphism $\psi_1$ in  \Cref{autom} sends $-x_0^2+x_1^2+x_2^2+x_3^2$ to $x_0^2+\a_2\a_3x_1^2-\a_3x_2^2+\a_2x_3^2$ so the latter also belongs
to the  center of $A(\a_1,\a_2,\a_3)$.  
\end{proof}

\begin{proposition}
\label{prop.moore}
Let $\Bbbk$ be any field.
If  $\a\b\c  \ne 0$ and $\a+\b+\c+\a\b\c  \ne 0$, then $x_0^2$,  $x_1^2$,  $x_2^2$, and  $x_3^2$, belong to the center of $A(\a,\b,\c)$.
\end{proposition}
\begin{proof}
We use the same notation as that in the proof of \Cref{prop.Skly.center}. Calculations show that if $(i,j,k)$ is a cyclic permutation of $(1,2,3)$, then 
\begin{align*}
&(\a_ j +\a_ k)  \{ x_i,c_i\} \;\; - \; \;  \a_ i (\a_ j \a_ k+1)[x_i,a_i]  \; \; - \; \;  \a_ i (\a_ k+1)  \big(\{ x_j,c_j\}  +[x_j,a_j] \big)
\\
 &  \phantom{xxxxxxxxxiixxxxxxxxxxxxxxxxxxxxx}  \; \; + \; \; 
 \a_ i (\a_ j-1)\big( \{ x_k,c_k\}  + [x_k,a_k] \big) 
 \\
&  \phantom{xxx} \; \;=\; \; (\a_ 1 + \a_ 2 + \a_ 3 + \a_ 1\a_ 2\a_ 3)[x_0,x_i^2], 
 \end{align*}
and
 \begin{align*}
&(\a_ j +\a_ k)  \{ x_j,a_i\} \;\; + \; \;   (\a_ j \a_ k+1)[x_j,c_i]  \; \; - \; \; (\a_k+1)\big(  \a_j\{x_i,a_j\} + [x_i,c_j] \big) 
\\
 &  \phantom{xxxxxxxxxiixxxxxxxxxxxxxxxxxxxx}  \; \; + \; \; 
(\a_j-1) \big( [x_0,a_k]  -\{x_0,c_k\}  \big) 
 \\
&  \phantom{xxx} \; \;=\; \; (\a_ 1 + \a_ 2 + \a_ 3 + \a_ 1\a_ 2\a_ 3)[x_k,x_j^2],
 \end{align*}
  and
  \begin{align*}
&( \alpha_i+\alpha_k)\{x_i,a_j\}   \;\; - \; \;    \alpha_i(\alpha_j\alpha_k+1)[x_i,c_j]   
  \; \; + \; \;  (\alpha_i+1)\big(\{x_0,c_k\}-[x_0,a_k] \big)  
\\
 &  \phantom{xxxxxxxxxiixxxxxxxxxxxxxxxxxxxx} 
  \; \; - \; \;   (\alpha_k-1)\big(\a_i\{x_j,a_i\}-[x_j,c_i]  \big) 
 \\
&  \phantom{xxx} \; \;=\; \; - \, (\a_ 1 + \a_ 2 + \a_ 3 + \a_ 1\a_ 2\a_ 3)[x_k,x_i^2]
 \end{align*}
 and
  \begin{align*}
& -(\alpha_j+\alpha_k)\{x_0,c_i\}       \;\; - \; \;    \a_i(\a_j\a_k+1)[x_0,a_i]     
  \; \; + \; \; \a_i (\alpha_k+1)\big(\a_j\{x_k,a_j\}-[x_k,c_j] \big)  
\\
 &  \phantom{xxxxxxxxxiixxxxxxxxxxxxxxxxxxxx} 
  \; \; + \; \;   \a_i(\alpha_j-1)\big(\a_k\{x_j,a_k\}+[x_j,c_k]  \big) 
 \\
&  \phantom{xxx} \; \;=\; \; - \,(\a_ 1 + \a_ 2 + \a_ 3 + \a_ 1\a_ 2\a_ 3)[x_i,x_0^2].
 \end{align*}

Since the images of $a_i$ and $c_i$ in $A(\a_1,\a_2,\a_3)$ are zero, if $\a_1+\a_2+\a_3+\a_1\a_2\a_3 \ne 0$, then 
$[x_0,x_1^2]=[x_0,x_2^2]=[x_0,x_3^2]=0$ and $[x_1,x_3^2]=[x_2,x_1^2]=[x_3,x_2^2]=0$ 
and $[x_1,x_2^2]=[x_2,x_3^2]=[x_3,x_1^2]=0$ in $A(\a_1,\a_2,\a_3)$. 
\end{proof}

\subsection{Degree-two central elements in $R(a,b,c,d)$}

It is conjectured at \cite[p. 47]{CHL} that the elements
$$
C_1:=ax_1x_3+bx_2x_4+cx_2^2+dx_1^2
$$
and 
\begin{equation*}
  C_2:=a'(v)(x_1x_3+x_3x_1)+b'(v)(x_2x_4+x_4x_2)+c'(v)(x_2^2+x_4^2)+d'(v)(x_1^2+x_3^2)
\end{equation*}
generate the center of $R(a,b,c,d)$ (we have suppressed an irrelevant scaling constant from the original expression of $C_2$ in \cite{CHL}). 
In order to have a little more symmetry, and to emphasize the parallels with the Sklyanin algebras, we will replace $C_1$ by 
\begin{equation*}
  Z_1:=a(x_1x_3+x_3x_1)+b(x_2x_4+x_4x_2)+c(x_2^2+x_4^2)+d(x_1^2+x_3^2), 
\end{equation*}
which is equal to $2C_1$, and replace $C_2$ by the element $Z_2$ in \Cref{cor.center} below, and show that $Z_1$ and $Z_2$ belong to the center of $R(a,b,c,d)$. 


\begin{proposition}\label{pr.O1}
  For all $a,b,c,d \in \Bbbk$, the element $Z_1$ is central in $R(a,b,c,d)$. 
\end{proposition} 
\begin{proof}
In terms of the generators $z_i$ in \Cref{prop.R.new.relns},
\begin{equation}
  \label{eq:O1}
Z_1 \; = \;2(b+c)z_0^2+2(a+d)z_1^2+2(-a+d)z_2^2+2(-b+c)z_3^2.  
\end{equation}
Using the expression for $Z_1$ in
\Cref{eq:O1}, we get 
  \begin{equation}
    \label{eq:O1_comm}
    [z_0,Z_1] \; = \; 2(a+d)[z_0,z_1^2] + 2(-a+d)[z_0,z_2^2] +2(-b+c)[z_0,z_3^2]. 
  \end{equation}
We now label the relations in the statement of \Cref{prop.R.new.relns} (in the form LHS $-$ RHS) according to which commutator or anticommutator involving $z_0$ they contain. For example, the first and third relations in \Cref{prop.R.new.relns}  are
$$
  c_{1}\;=\; (a-b-c+d)[z_0,z_1]-(-a-b+c+d)\{z_2,z_3\} \; = \; 0
  $$
  and
  $$
  a_{2}\;=\; (a+b+c-d)\{z_0,z_2\}-(-a+b-c-d)[z_3,z_1] \;= \; 0.
$$
With this in place, we leave the reader to check that \Cref{eq:O1_comm} equals 
\begin{equation*}
  \{z_1,c_{1}\}-[z_1,a_{1}]+\{z_2,c_{2}\}+[z_2,a_{2}]-2\{z_3,c_{3}\}-2[z_3,a_{3}],
\end{equation*}
which obviously belongs to the ideal generated by the relations $c_{i}$ and $a_{i}$. Thus $    [z_0,Z_1]=0$. 

We now prove that $[Z_1,z_i]=0$ for $i=1,2,3$ by changing the labels of the $z_i$ and the structure constants $a$, $b$, etc. so that both
 $Z_1$ and the space of relations in \Cref{prop.R.new.relns} are preserved. 
The transformation
\begin{equation*}
  z_0\longleftrightarrow z_1,\quad z_2\longleftrightarrow z_3,\quad a\longleftrightarrow b,\quad c\longleftrightarrow d
\end{equation*}
is such a relabeling so the fact that $[Z_1,z_0]=0$ implies $[Z_1,z_1]=0$. 
The transformation
\begin{equation*}
  z_0\longleftrightarrow z_3,\quad z_1\longleftrightarrow z_2,\quad a\longleftrightarrow -a,\quad b\longleftrightarrow -b  
\end{equation*}
(while $c$ and $d$ are fixed) is another such transformation, so  the fact that $[Z_1,z_0]=0$ implies $[Z_1,z_3]=0$. 
Finally, composing the two transformations will prove that $[Z_1,z_2]=0$. 
\end{proof}

\begin{proposition}
\label{prop.autom^4=1}
Let 
 $$
\rho_2   \;=\;  \frac{(a+b-c+d)(-a+b-c-d)}{(-a+b+c+d) (a+b+c-d)}
  \qquad \hbox{and} \qquad    \rho_3 \; = \;   \frac{da}{bc}.
$$
Assume that the denominators in the  expressions for $\rho_2$ and $\rho_3$ are non-zero. 
Fix $q_2$ and $q_3$ such that $q_2^4=\rho_2$ and $q_3^4=\rho_3$, and define 
$$
\tau_0:=-q_2q_3, \qquad \tau_1:=1/q_2q_3, \qquad \tau_2:= q_2/q_3, \qquad \tau_3:=q_3/q_2.
$$
The linear map $\psi:R_1 \to R_1$ given by the formula  
$$
\psi(z_0)=\tau_0 z_1, \quad \psi(z_1)=\tau_1 z_0, \quad \psi(z_2)=\tau_2 z_3, \quad \psi(z_3)=\tau_3 z_2,
$$ 
extends to an algebra automorphism of $R(a,b,c,d)$.
 \end{proposition}
 \begin{proof} 
 By \Cref{prop.R.new.relns}, $R$ is $\Bbbk\langle z_0,z_1,z_2,z_3\rangle$ modulo the relations 
$$
c_i   \;=\; \kappa_i[z_0,z_i]-\mu_i\{z_j,z_k\}  \qquad \hbox{and} \qquad
 a_i  \; = \; \l_i\{z_0,z_i\}-\nu_i[z_j,z_k],
 $$
  where $(i,j,k)$ runs over the cyclic permutation of $(1,2,3)$ and
  \begin{align*}
 \kappa_1&=a-b-c+d,  \quad \phantom{ii} \mu_1=-a-b+c+d, \quad \l_1=a+b+c+d, \quad \nu_1=a-b+c-d,   \\
  \kappa_2&= -a+b+c+d,    \quad \mu_2=a+b-c+d, \quad  \phantom{ii} \l_2=a+b+c-d, \quad \nu_2=-a+b-c-d, \\
   \kappa_3 &= b, \phantom{xxxxiixxxxx} \quad \mu_3= d, \phantom{xiixiiiiiixxix} \quad \l_3= c, \phantom{xxxiixxixx} \quad \nu_3= a.
\end{align*}
Furthermore, 
$$
\frac{ \tau_0\tau_1}{ \tau_2\tau_3}  \; = \; - 1,
\qquad
\frac{ \tau_0\tau_2}{ \tau_3\tau_1}   \; = \;  - \rho_2 \; = \; - \frac{\mu_2\nu_2}{\kappa_2\l_2} ,
\qquad
\frac{ \tau_0\tau_3}{  \tau_1\tau_2}   \; = \;  \rho_3  \; = \;  \frac{\mu_3\nu_3}{\kappa_3\l_3}. 
$$

Since
\begin{align*}
\psi(c_1) & \; = \; \kappa_1 \tau_0\tau_1[z_1,z_0]- \mu_1 \tau_2\tau_3\{z_3,z_2\}, 
\qquad
\psi(a_1)  \; = \; \l_1 \tau_0\tau_1\{z_1,z_0\}- \nu_1 \tau_2\tau_3[z_3,z_2],
\\
\psi(c_2) & \; = \; \kappa_2 \tau_0\tau_2[z_1,z_3]- \mu_2 \tau_3\tau_1\{z_2,z_0\} ,
\qquad
\psi(a_2)  \; = \; \l_2 \tau_0\tau_2\{z_1,z_3\}- \nu_2 \tau_3\tau_1[z_2,z_0],
\\
\psi(c_3) & \; = \; \kappa_3 \tau_0\tau_3[z_1,z_2]- \mu_3 \tau_1\tau_2\{z_0,z_3\},
\qquad
\psi(a_3)  \; = \; \l_3 \tau_0\tau_3\{z_1,z_2\}- \nu_3 \tau_1\tau_2[z_0,z_3],
\end{align*} 
we have
\begin{align*}
 \tau_2^{-1}\tau_3^{-1}  \psi(c_1) & \; = \; -\kappa_1[z_1,z_0]- \mu_1\{z_3,z_2\}   \phantom{xx}   =\;  c_1, 
\\
 \tau_2^{-1}\tau_3^{-1} \psi(a_1) &  \; = \; - \l_1 \{z_1,z_0\}- \nu_1 [z_3,z_2] \phantom{iix}   = \; -a_1,
\\
 \tau_3^{-1}\tau_1^{-1}\psi(c_2) & \; = \; -\kappa_2 \rho_2[z_1,z_3]- \mu_2\{z_2,z_0\} \; = \; - \frac{\mu_2}{\l_2}a_2 ,
\\
 \tau_3^{-1}\tau_1^{-1} \psi(a_2) & \; = \; - \l_2  \rho_2\{z_1,z_3\}- \nu_2 [z_2,z_0]  \phantom{i}  = \; \frac{\nu_2}{\kappa_2}c_2,
\\
\tau_1^{-1}\tau_2^{-1} \psi(c_3) & \; = \; \kappa_3 \rho_3[z_1,z_2]- \mu_3 \{z_0,z_3\}   \phantom{ix}   =  \; - \frac{\mu_3}{\l_3}a_3,
\\
\tau_1^{-1}\tau_2^{-1} \psi(a_3) & \; = \; \l_3\rho_3\{z_1,z_2\}- \nu_3 [z_0,z_3]  \phantom{xx}  =   \; -   \frac{\nu_3}{\kappa_3}c_3.
\end{align*} 
Hence $\psi$ extends to an algebra automorphism, as claimed. 

Since $\psi^2(z_0)=\tau_0\tau_1 z_0 =- z_0$, $\psi^2 \ne \id_R$. Since  $(\tau_0\tau_1)^2 =(\tau_2\tau_3)^2=1$, $\psi^4 = \id_R$. 
\end{proof}
 
\begin{corollary}
\label{cor.center}
With the notation and hypotheses in \Cref{prop.autom^4=1}, The element
$$
Z_2 \; := \; (a+d)(q_2q_3)^{-2}z_0^2 + (b+c)(q_2q_3)^{2}z_1^2 + (c-b)(q_2/q_3)^{2}z_2^2 + (d-a)(q_3/q_2)^{2}z_3^2 
$$
belongs to the center of $R$.
\end{corollary}
\begin{proof}
  Let $\psi$ be the automorphism in \Cref{prop.autom^4=1}.  Since
  $Z_2=\psi\big(\frac{1}{2}Z_1\big)$, the result follows from
  \Cref{pr.O1}.
 \end{proof}

\bibliography{biblio}
\bibliographystyle{plain}

\end{document}